\documentclass{amsart}
\usepackage{amsfonts,amssymb,amsmath,amsthm}
\usepackage{url}
\usepackage{enumerate}

\newtheorem{theorem}{Theorem}[section]
\newtheorem{lemma}[theorem]{Lemma}

\theoremstyle{definition}
\newtheorem{definition}[theorem]{Definition}
\newtheorem{remark}[theorem]{Remark}
\numberwithin{equation}{section}
\begin{document}

\title[commutator of the Marcinkiewicz integral]
{Weighted complete continuity for the commutator of  Marcinkiewicz integral}
\author[G. Hu]{Guoen Hu}
\address{Guoen Hu, Department of  Applied Mathematics, Zhengzhou Information Science and Technology Institute\\
Zhengzhou 450001,
P. R. China}

\email{guoenxx@163.com}
\thanks{The research  was supported by the NNSF of
China under grant $\#$11371370}.


\begin{abstract}
Let $\Omega$ be homogeneous of degree zero, have mean value zero and integrable on the unit sphere, and $\mathcal{M}_{\Omega}$ be the higher-dimensional Marcinkiewicz
integral associated with $\Omega$. In this paper, the author considers the complete continuity on  weighted $L^p(\mathbb{R}^n)$ spaces with $A_p(\mathbb{R}^n)$ weights, weighted  Morrey spaces with $A_p(\mathbb{R}^n)$ weights, for the  commutator generated by ${\rm CMO}(\mathbb{R}^n)$ functions  and
$\mathcal{M}_{\Omega}$ when $\Omega$ satisfies certain size conditions.
\end{abstract}
\maketitle

\section{Introduction}
As an analogy of the classicial Littlewood-Paley $g$-function,
Marcinkiewicz \cite{mar} introduced the operator
$$\mathcal{M}(f)(x)=\Big(\int^{\pi}_0\frac{|F(x+t)-F(x-t)-2F(x)|^2}{t^3}\,{\rm
d}t\Big)^{\frac{1}{2}}, $$ where $F(x)=\int^x_0f(t){\rm d}t.$ This operator
is now called Marcinkiewicz integral. Zygmund \cite{zy} proved that
$\mathcal{M}$ is bounded on  $L^p([0,\,2\pi])$ for $p\in
(1,\,\infty)$. Stein \cite{st} generalized the Marcinkiewicz
operator to the case of higher dimension. Let $\Omega$ be
homogeneous of degree zero, integrable and have mean value zero on
the unit sphere $S^{n-1}$. Define the Marcinkiewicz integral
operator $\mathcal{M}_\Omega$ by
\begin{eqnarray}\mathcal{M}_\Omega(f)(x)= \Big(\int_0^\infty|F_{\Omega,
t}f(x)|^2\frac{{\rm d}t}{t^3}\Big)^{\frac{1}{2}},\end{eqnarray}where $$F_{\Omega,
t}f(x)= \int_{|x-y|\leq t}\frac{\Omega(x-y)}{|x-y|^{n-1}}f(y){\rm d}y $$ for
$f\in \mathcal{S}(\mathbb{R}^n)$.  Stein \cite{st}
proved that if $\Omega\in {\rm Lip}_{\alpha}(S^{n-1})$ with
$\alpha\in (0,\,1]$, then $\mathcal{M}_\Omega$ is bounded on
$L^p(\mathbb{R}^n)$ for $p\in (1,\,2]$. Benedek, Calder\'on and
Panzon showed that the $L^p(\mathbb{R}^n)$ boundedness $(p\in
(1,\,\infty)$) of $\mathcal{M}_\Omega$ holds true  under the
condition that $\Omega\in C^1(S^{n-1})$. Using the one-dimensional
result and Riesz transforms similarly as in the case of singular
integrals (see \cite{cz}) and interpolation, Walsh \cite{wal} proved
that for each $p\in (1,\,\infty)$, $\Omega\in L(\ln L)^{1/r}(\ln \ln
L)^{2(1-2/r')}(S^{n-1})$ is a sufficient condition such that
$\mathcal{M}_\Omega$ is bounded on $L^{p}(\mathbb{R}^n)$, where
$r=\min\{p,\,p'\}$ and $p'=p/(p-1)$. Ding, Fan and Pan \cite{dfp}
proved that if $\Omega\in H^1(S^{n-1})$ (the Hardy space on
$S^{n-1}$), then $\mathcal{M}_\Omega$ is bounded on
$L^p(\mathbb{R}^n)$  for all $p\in (1,\,\infty)$; Al-Salmam,  Al-Qassem,  Cheng and  Pan
\cite{aacp} proved that $\Omega\in L(\ln L)^{1/2}(S^{n-1})$ is a
sufficient condition such that $\mathcal{M}_\Omega$ is bounded on
$L^p(\mathbb{R}^n)$ for all $p\in(1,\,\infty)$. Ding, Fan and Pan \cite{dfp2} considered the boundedness on weighted $L^p({\mathbb R}^n)$
with $A_p(\mathbb{R}^n)$ when $\Omega\in L^q(S^{n-1})$ for some $q\in (1,\,\infty]$, where and in the following, for $p\in [1,\,\infty)$, $A_p(\mathbb{R}^n)$ denotes the weight function class of Muckenhoupt, see \cite{gra2} for the definitions and properties of $A_p(\mathbb{R}^n)$. For other works about the operator defined by (1.1), see \cite{a,aacp, cfp, dfp, dly1,dxy} and the related references therein.

The commutator of  $\mathcal{M}_{\Omega}$ is also of interest and
has been considered by many authors (see \cite{tw, hy,dly,chenlu,h2}). Let $b\in {\rm BMO}(\mathbb{R}^n)$, the commutator generated by
$\mathcal{M}_\Omega$ and $b$ is defined by
\begin{eqnarray}\qquad\mathcal{M}_{\Omega, b}f(x)=\bigg(\int_0^\infty\Big|\int_{|x-y|\leq t}\big(b(x)-b(y)\big)
\frac{\Omega(x-y)}{|x-y|^{n-1}}f(y)dy\Big|^2\frac{dt}{t^3}\bigg)^{\frac{1}{2}}.
\end{eqnarray} Torchinsky and Wang \cite{tw} showed that if $\Omega\in
{\rm Lip}_\alpha(S^{n-1})$ ($\alpha\in (0,\,1]$), then
$\mathcal{M}_{\Omega, b}$ is bounded on $L^p(\mathbb{R}^n)$ with
bound $C\|b\|_{{\rm BMO}(\mathbb{R}^n)}$ for all $p\in
(1,\,\infty)$. Hu and Yan \cite{hy} proved the $\Omega\in L(\ln
L)^{3/2}(S^{n-1})$ is a sufficient condition such that
$\mathcal{M}_{\Omega,\,b}$ is  bounded on $L^2$.  Ding, Lu and Yabuta \cite{dly} considered the weighted estimates for $\mathcal{M}_{\Omega,\,b}$, and proved that
if $\Omega\in L^{q}(S^{n-1})$ for some $q\in (1,\,\infty]$, then for $p\in (q',\,\infty)$ and $w\in A_{p/q'}(\mathbb{R}^n)$, or $p\in (1,\,q)$ and $w^{-1/(p-1)}\in A_{p'/q'}(\mathbb{R}^n)$, $\mathcal{M}_{\Omega,\,b}$ is bounded on $L^p(\mathbb{R}^n,\,w)$.  Chen and Lu \cite{chenlu}
improved the result in \cite{hy} and showed that if $\Omega\in L(\ln
L)^{3/2}(S^{n-1})$, then $\mathcal{M}_{\Omega,\,b}$ is bounded on
$L^p(\mathbb{R}^n)$ with bound  $C\|b\|_{{\rm BMO}(\mathbb{R}^n)}$
for all $p\in (1,\,\infty)$.

Let ${\rm CMO}(\mathbb{R}^n)$ be the
closure of $C^{\infty}_0(\mathbb{R}^n)$ in the ${\rm
BMO}(\mathbb{R}^n)$ topology, which coincide with ${\rm VMO}(\mathbb{R}^n)$, the space of
functions of vanishing mean oscillation introduced by Coifman and Weiss \cite{cw}, see also \cite{b}. Uchiyama \cite{u}
proved that if $T$ is a Calder\'on-Zygmund operator, and $b\in {\rm
BMO}(\mathbb{R}^n)$, then the commutator of $T$ defined by
$$[b,\,T]f(x)=b(x)Tf(x)-T(bf)(x),$$ is a compact operator on $L^p(\mathbb{R}^n)$ $(p\in
(1,\,\infty)$) if and only if $b\in {\rm CMO}(\mathbb{R}^n)$.
Chen and  Ding \cite{chend} considered the compactness of
$\mathcal{M}_{\Omega,\,b}$ on $L^p(\mathbb{R}^n)$, and proved that
if $\Omega$ satisfies certain regularity condition of Dini type,
then for $p\in (1,\,\infty)$, $\mathcal{M}_{\Omega,\,b}$ is compact on
$L^p(\mathbb{R}^n)$ if and only if $b\in {\rm CMO}(\mathbb{R}^n)$. Using the ideas from \cite{chenhu}, Mao, Sawano and Wu \cite{msw} consider the compactness of $\mathcal{M}_{\Omega,\,b}$ when $\Omega$ satisfies the size condition that
\begin{eqnarray} \sup_{\zeta\in
S^{n-1}}\int_{S^{n-1}}|\Omega(\eta)|\Big(\ln\frac{1}{|\eta\cdot\zeta|}\Big)^{\theta}{\rm
d}\eta<\infty,
\end{eqnarray}
and proved that if $\Omega$ satisfies (1.3) for some $\theta\in (3/2,\,\infty)$, then for $b\in {\rm CMO}(\mathbb{R}^n)$ and $p\in \big(4\theta/(4\theta-3),\,4\theta/3\big)$, $\mathcal{M}_{\Omega,\,b}$ is compact on $L^p(\mathbb{R}^n)$. Recently,  Chen and Hu \cite{chenhu2} proved that if $\Omega\in L(\log L)^{\frac{1}{2}}(S^{n-1})$, then for $b\in {\rm CMO}(\mathbb{R}^n)$ and $p\in (1,\,\infty)$, $\mathcal{M}_{\Omega,\,b}$ is completely continuous on $L^p(\mathbb{R}^n)$. Our
first purpose of this paper is to consider the complete continuity on weighted $L^p(\mathbb{R}^n)$ spaces $\mathcal{M}_{\Omega,\,b}$ when $\Omega\in L^q(S^{n-1})$ for some $q\in (1,\,\infty]$. To formulate
our main result, we first recall some definitions.

\begin{definition}
 Let $\mathcal{X}$ be a normed linear spaces and
$\mathcal{X}^{*}$ be its dual space, $\{x_k\}\subset \mathcal{X}$
and $x\in \mathcal{X}$, If for all $f\in \mathcal{X}^*$,
$$\lim_{k\rightarrow \infty}|f(x_k)-f(x)|=0,$$
then $\{x_k\}$ is said to converge to $x$ weakly, or
$x_k\rightharpoonup x$.
\end{definition}

\begin{definition}
Let $\mathcal{X}$, $\mathcal{Y}$ be two Banach spaces and $S$ be a
bounded operator from $\mathcal{X}$ to $\mathcal{Y}$.
\begin{itemize}
\item[\rm (i)] If for each bounded set $\mathcal{G}\subset \mathcal{X}$,
$S\mathcal{G}=\{Sx:\,x\in \mathcal{G}\}$ is  a strongly pre-compact
set in $\mathcal{Y}$, then $S$ is called a compact operator from
$\mathcal{X}$ to $\mathcal{Y}$;
\item[\rm (ii)] if for $\{x_k\}\subset\mathcal{X}$ and $x\in\mathcal{X}$,
$$x_k\rightharpoonup x\,\,\hbox{in}\, \,\mathcal{X}\Rightarrow \|Sx_k-Sx\|_{\mathcal{Y}}\rightarrow 0,$$
then $S$ is said to be a completely continuous operator.
\end{itemize}
\end{definition}

It is well known that, if $\mathcal{X}$ is a reflexive space, and
$S$ is completely continuous from $\mathcal{X}$ to $\mathcal{Y}$,
then $S$ is also compact from $\mathcal{X}$ to $\mathcal{Y}$. On the
other hand, if $S$  is a linear compact operator from $\mathcal{X}$
to $\mathcal{Y}$, then $S$ is also a completely continuous operator.
However, if $S$ is not linear, then $S$ is compact  do not
imply that $S$ is completely continuous. For example, the operator
$$Sx=\|x\|_{l^2}$$
is compact from $l^2$ to $\mathbb{R}$, but  not completely
continuous.

Our first result in this paper can be stated as follows.

\begin{theorem}\label{t1.2}
Let $\Omega$ be homogeneous of degree zero, have mean value zero on $S^{n-1}$ and $\Omega\in L^q(S^{n-1})$
for some $q\in (1,\,\infty]$.
Suppose that  $p$ and $w$ satisfy one of the following conditions
\begin{itemize}
\item[\rm (i)] $p\in (q',\,\infty)$ and $w\in A_{p/q'}(\mathbb{R}^n)$;
\item[\rm (ii)] $p\in (1,\,q)$ and $w^{-1/(p-1)}\in
A_{p'/q'}(\mathbb{R}^n)$;
\item[\rm (iii)] $p\in (1,\,\infty)$ and $w^{q'}\in A_{p}(\mathbb{R}^n)$.
\end{itemize} Then for $b\in {\rm
CMO}(\mathbb{R}^n)$, $\mathcal{M}_{\Omega,\,b}$ is completely
continuous on $L^p(\mathbb{R}^n,\,w)$.
\end{theorem}

Our argument used in the proof of Theorem \ref{t1.2} also leads to the complete continuity of $\mathcal{M}_{\Omega,\,b}$ on weighted Morrey spaces.
\begin{definition} Let $p\in (0,\,\infty)$, $w$ be a weight and $\lambda\in (0,\,1)$. The
weighted Morrey space $L^{p,\,\lambda}(\mathbb{R}^n,\,w)$ is defined as
$$L^{p,\,\lambda}(\mathbb{R}^n,\,w)=\{f\in L_{\rm loc}^p(\mathbb{R}^n):\,\|f\|_{L^{p,\,\lambda}(\mathbb{R}^n,\,w)}<\infty\},$$
with $$\|f\|_{L^{p,\,\lambda}(\mathbb{R}^n,\,w)}=\sup_{y\in
\mathbb{R}^n,\,r>0}\Big(\frac{1}{\{w(B(y,\,r))\}^{\lambda}}\int_{B(y,\,r)}|f(x)|^pw(x)\,{\rm
d}x\Big)^{1/p},$$¡¡
here $B(y,\,r)$ denotes
the ball in $\mathbb{R}^n$ centered at $y$ and having radius $r$,
and $w(B(y,\,r))=\int_{B(y,\,r)}w(z){\rm d}z$. For simplicity, we use $L^{p,\,\lambda}(\mathbb{R}^n)$ to denote  $L^{p,\,\lambda}(\mathbb{R}^n,1)$.
\end{definition}
The Morrey space  $L^{p,\,\lambda}(\mathbb{R}^n)$ was introduced by Morrey
[17]. It is well-known that this space is closely related to some
problems in PED (see \cite{rv, shen1}), and has interest in harmonic
analysis (see \cite{ax} and the references therein). Komori and
Shiral \cite{ks} introduced the weighted Morrey spaces and
considered the properties on weighted Morrey spaces for some
classical operators. Chen, Ding and Wang \cite{cdw} considered the
compactness of $\mathcal{M}_{\Omega,\,b}$ on Morrey spaces. They proved that
if $\lambda\in (0,\,1)$, $\Omega\in L^q(S^{n-1})$ for $q\in
(1/(1-\lambda),\,\infty]$ and satisfies a regularity condition of $L^q$-Dini type, then $\mathcal{M}_{\Omega,\,b}$ is compact on $L^{p,\,\lambda}(\mathbb{R}^n)$. Our second purpose of this paper is to prove
the complete continuity of $\mathcal{M}_{\Omega,\,b}$ on weighted Morrey spaces with $A_p(\mathbb{R}^n)$ weights.
\begin{theorem}\label{t1.3}
Let $\Omega$ be homogeneous of degree zero,  have mean value zero
on $S^{n-1}$ and $\Omega\in L^q(S^{n-1})$ for some $q\in
(1,\,\infty]$. Suppose that $p\in (q',\,\infty)$, $\lambda\in (0,\,1)$ and $w\in
A_{p/q'}(\mathbb{R}^n)$; or $p\in (1,\,q')$,  $w^{r}\in A_1(\mathbb{R}^n)$ for some $r\in (q',\,\infty)$ and $\lambda\in
(0,\,1-r'/q)$. Then for $b\in {\rm
CMO}(\mathbb{R}^n)$, $\mathcal{M}_{\Omega,\,b}$ is completely
continuous on $L^{p,\,\lambda}(\mathbb{R}^n,\,w)$.
\end{theorem}

\begin{remark} The proofs of Theorems \ref{t1.2} involve some ideas used in  \cite{chenhu} and a sufficient condition of strongly pre-compact set in $L^p(L^2([1,\,2]), l^2;\,\mathbb{R}^n,\,w)$ with $w\in A_p(\mathbb{R}^n)$. To prove Theorem \ref{t1.3}, we will establish a lemma which clarify  the relationship of the bounds on $L^p(\mathbb{R}^n,\,w)$ and the bounds on $L^{p,\,\lambda}(\mathbb{R}^n,\,w)$ for a class of sublinear operators, see Lemma 4.1 below.
\end{remark}

 We make some conventions. In what follows, $C$ always denotes a
positive constant that is independent of the main parameters
involved but whose value may differ from line to line. We use the
symbol $A\lesssim B$ to denote that there exists a positive constant
$C$ such that $A\le CB$.   For a set $E\subset\mathbb{R}^n$,
$\chi_E$ denotes its characteristic function.  Let $M$ be the Hardy-Littlewood
maximal operator. For $r\in (0,\,\infty)$, we use $M_r$ to denote
the operator $M_rf(x)=\big(M(|f|^r)(x)\big)^{1/r}.$ For a locally integrable function $f$, the sharp maximal function $M^{\sharp}f$ is defined by
$$M^{\sharp}f(x)=\sup_{Q\ni x}\inf_{c\in\mathbb{C}}\frac{1}{|Q|}\int_{Q}|f(y)-c|{\rm d}y.$$

\section{Approximation}
Let $\Omega$ be homogeneous of degree zero, integrable on $S^{n-1}$. For $t\in [1,\,2]$ and $j\in\mathbb{Z}$, set
\begin{eqnarray}K^j_t(x)=\frac{1}{2^j}\frac{\Omega(x)}{|x|^{n-1}}\chi_{\{2^{j-1}t<|x|\leq 2^jt\}}(x).\end{eqnarray}
As it was proved in \cite{drf}, if $\Omega\in L^q(S^{n-1})$ for some
$q\in (1,\,\infty]$, then there exists a constant $\alpha\in
(0,\,1)$ such that for $t\in[1,\,2]$ and
$\xi\in\mathbb{R}^n\backslash\{0\}$,
\begin{eqnarray}|\widehat{K^j_t}(\xi)|\lesssim \|\Omega\|_{L^{q}(S^{n-1})}\min\{1,\,|2^j\xi|^{-\alpha}\}.\end{eqnarray} Here and in the following for $h\in\mathcal{S}'(\mathbb{R}^n)$, $\widehat{h}$ denotes the Fourier transform of $h$. Moreover, if $\int_{S^{n-1}}\Omega(x'){\rm d}x'=0$, then
\begin{eqnarray}|\widehat{K^j_t}(\xi)|\lesssim \|\Omega\|_{L^{1}(S^{n-1})}\min\{1,\,|2^j\xi|\}.\end{eqnarray}

Let
$$\widetilde{\mathcal{M}}_{\Omega}f(x)=\Big(\int^2_1\sum_{j\in\mathbb{Z}}\big|F_{j}f(x,\,t)\big|^2{\rm d}t\Big)^{\frac{1}{2}},$$with
$$F_jf(x,\,t)=\int_{\mathbb{R}^n}K^j_t(x-y)f(y){\rm d}y.$$
For $b\in {\rm BMO}(\mathbb{R}^n)$, let
$\widetilde{\mathcal{M}}_{\Omega,\,b}$ be the commutator of
$\widetilde{\mathcal{M}}_{\Omega}$ defined by
$$\widetilde{\mathcal{M}}_{\Omega,\,b}f(x)=\Big(\int^2_1\sum_{j\in\mathbb{Z}}\big|F_{j,\,b}f(x,\,t)\big|^2{\rm d}t\Big)^{1/2},$$with
$$F_{j,\,b}f(x,\,t)=\int_{\mathbb{R}^n}\big(b(x)-b(y)\big)K^j_t(x-y)f(y){\rm d}y.$$
A trivial computation leads to that
\begin{eqnarray}\mathcal{M}_{\Omega}f(x)\approx \widetilde{\mathcal{M}}_{\Omega}f(x),\,\,
\widetilde{\mathcal{M}}_{\Omega,\,b}f(x)\approx\widetilde{\mathcal{M}}_{\Omega,\,b}f(x).
\end{eqnarray}

 Let $\phi\in
C^{\infty}_0(\mathbb{R}^n)$ be a nonnegative function such that
$\int_{\mathbb{R}^n}\phi(x){\rm d}x=1$, ${\rm
supp}\,\phi\subset\{x:\,|x|\leq 1/4\}$. For $l\in \mathbb{Z}$, let
$\phi_l(y)=2^{-nl}\phi(2^{-l}y)$. It is easy to verify that for any
$\varsigma\in (0,\,1)$,
\begin{eqnarray}|\widehat{\phi_l}(\xi)-1|\lesssim \min\{1,\,|2^l\xi|^{\varsigma}\}.\end{eqnarray}Let
$$F_{j}^lf(x,\,t)=\int_{\mathbb{R}^n}K^j_t*\phi_{j-l}(x-y)f(y)\,{\rm d}y.
$$
Define the operator $\widetilde{\mathcal{M}}_{\Omega}^l$ by
\begin{eqnarray}\widetilde{\mathcal{M}}_{\Omega}^lf(x)=\Big(\int^2_1\sum_{j\in\mathbb{Z}}\big|F_{j}^lf(x,\,t)\big|^2
{\rm d}t\Big)^{\frac{1}{2}}.\end{eqnarray}

This section is devoted to the approximation of $\widetilde{\mathcal{M}}_{\Omega}$ by $\widetilde{\mathcal{M}}_{\Omega}^l$. We will prove following theorem.

\begin{theorem}\label{t2.2}
Let $\Omega$ be homogeneous of degree zero and have mean value zero.
Suppose that $\Omega\in L^q(S^{n-1})$ for some $q\in (1,\,\infty]$,
$p$ and $w$ are the same as in Theorem \ref{t1.2}, then for
$l\in\mathbb{N}$,
$$\|\widetilde{\mathcal{M}}_{\Omega}f-\widetilde{\mathcal{M}}_{\Omega}^{l}f\|_{L^p(\mathbb{R}^n,\,w)}\lesssim
2^{-\varrho_pl}\|f\|_{L^{p}(\mathbb{R}^n,\,w)},$$ with $\varrho_p\in
(0,\,1)$ a constant depending only on $p$,\,$n$ and $w$.
\end{theorem}

To prove Theorem \ref{t2.2}, we will use some lemmas.

\begin{lemma}\label{l2.2}
Let $\Omega$ be
homogeneous of degree zero and belong to $L^q(S^{n-1})$ for some $q\in (1,\,\infty]$, $K_t^j$ be defined as in (2.1). Then for $t\in [1,\,2]$,
$l\in\mathbb{N}$, $R>0$ and  $y\in \mathbb{R}^n$ with $|y|<R/4$,
\begin{eqnarray*}\sum_{j\in\mathbb{Z}}\sum_{k=1}^{\infty}(2^kR)^{\frac{n}{q'}}\Big(\int_{2^kR<|x|\leq 2^{k+1}R}\big|K^j_{t}*\phi_{j-l}(x+y)-K^j_{t}*\phi_{j-l}(x)\big|^{q}{\rm d}x\Big)^{\frac{1}{q}}\lesssim l.\end{eqnarray*}
\end{lemma}
For the proof of Lemma \ref{l2.2}, see \cite{wat}.

\begin{lemma} \label{l2.4}Let $\Omega$ be homogeneous of degree zero and $\Omega\in L^q(S^{n-1})$ for some $q\in (1,\,\infty]$, $p\in (1,\,q)$ and $w^{-1/(p-1)}\in A_{p'/q'}(\mathbb{R}^n)$. Then
\begin{eqnarray}\quad\Big\|\Big(\sum_{j\in\mathbb{Z}}|K^j_t*\phi_{j-l}*f_j|^2\Big)^{\frac{1}{2}}\Big\|_{L^p(\mathbb{R}^n,\,w)}\lesssim
\Big\|\Big(\sum_{j\in\mathbb{Z}}|f_j|^2\Big)^{\frac{1}{2}}\Big\|_{L^p(\mathbb{R}^n,\,w)}.\end{eqnarray}
\end{lemma}

\begin{proof} Let $M_{\Omega}$ be the maximal operator defined by
\begin{eqnarray}M_{\Omega}h(x)=\sup_{r>0}\frac{1}{|B(x,\,r)|}\int_{B(x,\,r)}|\Omega(x-y)h(y)|{\rm d}y.\end{eqnarray}
We know from the proof of Lemma 1 in \cite{duo} that for $p\in (1,\,2]$,
\begin{eqnarray}\Big\|\Big(\sum_{j\in\mathbb{Z}}|M_{\Omega}f_j|^2\Big)^{\frac{1}{2}}\Big\|_{L^p(\mathbb{R}^n,\,w)}\lesssim
\Big\|\Big(\sum_{j\in\mathbb{Z}}|f_j|^2\Big)^{\frac{1}{2}}\Big\|_{L^p(\mathbb{R}^n,\,w)},\end{eqnarray}
provided that $p\in (q',\,\infty)$ and $w\in A_{p/q'}(\mathbb{R}^n)$, or $p\in (1,\,q)$ and $w^{-1/(p-1)}\in A_{p'/q'}(\mathbb{R}^n)$. On the other hand, it is easy to verify that
$$|K^j_t*\phi_{j-l}*f_j(x)|\lesssim M_{\Omega}Mf_j(x).
$$
The inequality (2.9), together with the weighted vector-valued inequality of $M$ (see Theorem 3.1 in \cite{aj}), proves that (2.7) hold when $p\in (1,\,2]$, $p\in (q',\,\infty)$ and $w\in A_{p/q'}(\mathbb{R}^n)$, or $p\in (1,\,q)$ and $w^{-1/(p-1)}\in A_{p'/q'}(\mathbb{R}^n)$. This, via  a standard duality argument, shows that (2.7) holds when $p\in(2,\,\infty)$, $p\in (1,\,q)$ and $w^{-1/(p-1)}\in A_{p'/q'}(\mathbb{R}^n)$.
\end{proof}

{\it Proof of Theorem \ref{t2.2}}. We employ the ideas used in \cite{wat}. By Fourier transform estimates (2.2) and (2.5), and the Plancherel theorem, we know
that
\begin{eqnarray*}
\|\widetilde{\mathcal{M}}_{\Omega}f-\widetilde{\mathcal{M}}_{\Omega}^{l}f\|_{L^2(\mathbb{R}^n)}^2
&=&\int^2_1\Big\|\Big(\sum_{j\in\mathbb{Z}}\big|F_lf(\cdot,\,t)-F_{j}^lf(\cdot,\,t)\big|^2\Big)^{\frac{1}{2}}\Big\|_{L^2(\mathbb{R}^n)}^2{\rm
d}t\\
&=&\int^2_1\sum_{j\in\mathbb{Z}}\int_{\mathbb{R}^n}|\widehat{K_t^j}(\xi)|^2|1-\widehat{\phi_{j-l}}(\xi)|^2|\widehat{f}(\xi)|^2{\rm
d}\xi{\rm d}t\\
&\lesssim&2^{-\alpha l}\|f\|_{L^2(\mathbb{R}^n)}^2.
\end{eqnarray*}
Now let $p$ and $w$ be the same as in Theorem \ref{t1.2}. Recall that $\mathcal{M}_{\Omega}$  is bounded on $L^p(\mathbb{R}^n,\,w)$ and so is  $\widetilde{\mathcal{M}}_{\Omega}$. Thus, by interpolation with changes of measures of Stein and Weiss \cite{stw}, it suffices to prove that
\begin{eqnarray}\|\widetilde{\mathcal{M}}_{\Omega}^{l}f\|_{L^p(\mathbb{R}^n,\,w)}\lesssim
l\|f\|_{L^{p}(\mathbb{R}^n,\,w)}.\end{eqnarray}

We now prove (2.10) for the case $p\in (1,\,q)$ and $w^{-1/(p-1)}\in A_{p'/q'}(\mathbb{R}^n)$. Let $\psi\in C^{\infty}_0(\mathbb{R}^n)$ be a radial function such that ${\rm supp}\,\psi\subset \{1/4\leq |\xi|\leq 4\}$ and
$$\sum_{i\in\mathbb{Z}}\psi(2^{-i}\xi)=1,\,\,|\xi|\not =0.$$
Define the multiplier operator $S_i$ by
$$\widehat{S_if}(\xi)=\psi(2^{-i}\xi)\widehat{f}(\xi).$$
Set
$${\rm E}_1f(x)=\sum_{m=-\infty}^0\Big(\int^2_1\sum_{j}\Big|K^j_t*\phi_{j-l}*(S_{m-j}f)(x)\Big|^2{\rm d}t\Big)^{\frac{1}{2}},$$
$${\rm E_2}f(x)=\sum_{m=1}^\infty\Big(\int^2_1\sum_{j}\Big|K^j_t*\phi_{j-l}*(S_{m-j}f)(x)\Big|^2{\rm d}t\Big)^{\frac{1}{2}}.$$
It then follows that for $f\in \mathcal{S}(\mathbb{R}^n)$,
$$\Big\|\Big(\int^2_1\sum_{j}\big|K^j_t*\phi_{j-l}*f(x)\big|^2{\rm d}t\Big)^{\frac{1}{2}}\Big\|_{L^p(\mathbb{R}^n)}\leq
\sum_{i=1}^2\|{\rm E}_if\|_{L^p(\mathbb{R}^n)}.$$

We now estimate the term ${\rm E}_1$. By Fourier transform estimate (2.3), we know that
\begin{eqnarray}
&&\Big\|\Big(\int^2_1\sum_{j}\Big|K^j_t*\phi_{j-l}*(S_{m-j}f)(x)\Big|^2{\rm d}t\Big)^{\frac{1}{2}}\Big\|_{L^2(\mathbb{R}^n)}^2\\
&&\quad=\int^2_1\int_{\mathbb{R}^n}\sum_{j\in\mathbb{Z}}\Big|K^j_t*\phi_{j-l}*(S_{m-j}f)(x)\Big|^2{\rm d}x{\rm d}t\nonumber\\
&&\quad\lesssim\sum_{j\in\mathbb{Z}}\int_{\mathbb{R}^n}|2^j\xi||\psi(2^{-m+j}\xi)|^2|\widehat{f}(\xi)|^2{\rm d}\xi\nonumber\\
&&\quad\leq 2^{2m}\|f\|_{L^2(\mathbb{R}^n)}^2.\nonumber
\end{eqnarray}
On the other hand,  applying the Minkowski
inequality, Lemma \ref{l2.4} and the weighted Littlewood-Paley theory, we have  that
\begin{eqnarray}
&&\Big\|\Big(\int^2_1\sum_{j}\Big|K^j_t*\phi_{j-l}*(S_{m-j}f)(x)\Big|^2{\rm d}t\Big)^{\frac{1}{2}}\Big\|_{L^p(\mathbb{R}^n,\,w)}^2\\
&&\quad\leq \int^2_1\Big(\int_{\mathbb{R}^n}\Big(\sum_{j\in\mathbb{Z}}\Big|K^j_t*\phi_{j-l}*(S_{m-j}f)(x)\Big|^2\Big)^{p/2}w(x){\rm d}x\Big)^{2/p}{\rm d}t\nonumber\\
&&\quad\leq \|f\|_{L^p(\mathbb{R}^n,\,w)}^2,\,\,p\in[2,\,\infty).\nonumber
\end{eqnarray}
To estimate
$$\Big\|\Big(\int^2_1\sum_{j}\Big|K^j_t*\phi_{j-l}*(S_{m-j}f)(x)\Big|^2{\rm
d}t\Big)^{\frac{1}{2}}\Big\|_{L^p(\mathbb{R}^n,\,w)}$$ for $p\in (1,\,2)$, we
consider the mapping $\mathcal{F}$ defined by
$$\mathcal{F}:\,\,\{h_j(x)\}_{j\in\mathbb{Z}}\longrightarrow \{K^j_t*\phi_{j-l}*h_j(x)\}.$$
Note that for any $t\in (1,\,2)$,
$$\big|K^j_t*\phi_{j-l}*h_j(x)\big|\lesssim MM_{\Omega}h_j(x).$$
We choose $p_0\in (1,\,p)$ such that $w^{-1/(p_0-1)}\in A_{p_0'/q'}(\mathbb{R}^n)$. Then by the weighted estimates for $M_{\Omega}$ (see \cite{duo}), we have  that
\begin{eqnarray}&&\int_{\mathbb{R}^n}\int^2_1\sum_{j\in \mathbb{Z}}\big|K^j_t*\phi_{j-l}*h_j(x)\big|^{p_0}{\rm d}tw(x){\rm d}x
\lesssim\int_{\mathbb{R}^n}\sum_{j\in\mathbb{Z}}|h_j(x)|^{p_0}w(x){\rm
d}x.
\end{eqnarray}
Also, we have that
$$\sup_{j\in \mathbb{Z}}\sup_{t\in[1,\,2]}\big|K^j_t*\phi_{j-l}*h_j(x)\big|\lesssim\sup_{j\in\mathbb{Z}}|h_j(x)|.
$$
which implies that for $p_1\in (1,\,\infty)$,
\begin{eqnarray}&&\Big\|\sup_{j\in \mathbb{Z}}\sup_{t\in[1,\,2]}\big|K^j_t*\phi_{j-l}*h_j\big|\Big\|_{L^{p_1}(\mathbb{R}^n,\,w)}\lesssim
\Big\|\sup_{j\in\mathbb{Z}}|h_j|\Big\|_{L^{p_1}(\mathbb{R}^n,\,w)}.
\end{eqnarray}
By interpolation, we deduce from the inequalities (2.13) and (2.14)
(with $p_0\in (1,\,2)$, $p_1\in (2,\,\infty)$ and
$1/p=1/2+(2-p_0)/(2p_1)$) that
$$\Big\|\Big(\int^2_1\sum_{j\in \mathbb{Z}}\big|K^j_t*\phi_{j-l}*h_j\big|^{2}{\rm d}t\Big)^{\frac{1}{2}}\Big\|_{L^p(\mathbb{R}^n,\,w)}\lesssim\Big\|\Big(\sum_{j\in\mathbb{Z}}|h_j|^{2}\Big)^{\frac{1}{2}}
\Big\|_{L^p(\mathbb{R}^n,\,w)},
$$
and so
\begin{eqnarray*}&&\Big\|\Big(\int^2_1\sum_{j}\Big|K^j_t*\phi_{j-l}*(S_{m-j}f)\Big|^2{\rm d}t\Big)^{\frac{1}{2}}
\Big\|_{L^p(\mathbb{R}^n,w)}\\
&&\quad\lesssim
\Big\|\Big(\sum_{j\in\mathbb{Z}}|S_{m-j}f|^2\Big)^{\frac{1}{2}}\Big\|_{L^p(\mathbb{R}^n,\,w)}
\lesssim
\|f\|_{L^p(\mathbb{R}^n,\,w)},\,p\in (1,\,2).\nonumber
\end{eqnarray*}
This, along with (2.12), states that  for $p\in (1,q)$,
\begin{eqnarray}\Big\|\Big(\int^2_1\sum_{j}\Big|K^j_t*\phi_{j-l}*(S_{m-j}f)\Big|^2{\rm d}t\Big)^{\frac{1}{2}}
\Big\|_{L^p(\mathbb{R}^n,w)}\lesssim
\|f\|_{L^p(\mathbb{R}^n,w)}.
\end{eqnarray}
Again by  interpolating,  the inequalities (2.11) and (2.15) give us that for
$p\in (1,\,q)$,
$$\Big\|\Big(\int^2_1\sum_{j}\Big|K^j_t*\phi_{j-l}*(S_{m-j}f)(x)\Big|^2{\rm d}t\Big)^{\frac{1}{2}}\Big\|_{L^p(\mathbb{R}^n,\,w)}\\
\lesssim 2^{t_p
m}\|f\|_{L^p(\mathbb{R}^n,\,w)}.
$$
with $t_p\in (0,\,1)$ a constant depending only on $p$. Therefore,
$$\|{\rm E}_1f\|_{L^p(\mathbb{R}^n,\,w)}\lesssim \|f\|_{L^p(\mathbb{R}^n,\,w)}.$$

We consider the term ${\rm E}_2$. Again by the
Plancherel theorem and the Fourier transform estimates (2.2) and
(2.5), we have that
\begin{eqnarray}
&&\Big\|\Big(\int^2_1\sum_{j\in\mathbb{Z}}\Big|K^j_t*\phi_{j-l}*(S_{m-j}f)(x)\Big|^2{\rm d}t\Big)^{\frac{1}{2}}\Big\|^2_{L^2(\mathbb{R}^n)}\\
&&\quad=\int^2_1\sum_{j\in\mathbb{Z}}\int_{\mathbb{R}^n}|\widehat{K_t^j}(\xi)|^2|\psi(2^{-m+j}\xi)|^2|\widehat{f}(\xi)|^2{\rm d}\xi{\rm d}t\nonumber\\
&&\quad\lesssim\sum_{j\in\mathbb{Z}}\int_{\mathbb{R}^n}|2^j\xi|^{-2\alpha}|2^{j-l}\xi|^{\alpha}\psi(2^{-m+j}\xi)|^2|\widehat{f}(\xi)|^2{\rm d}\xi\nonumber\\
&&\quad\lesssim2^{-m\alpha}\|f\|_{L^2(\mathbb{R}^n)}^2.\nonumber
\end{eqnarray}
As in the inequality (2.15), we have that
\begin{eqnarray}
\qquad\Big\|\Big(\int^2_1\sum_{j\in\mathbb{Z}}\Big|K^j_t*\phi_{j-l}*(S_{m-j}f)(x)\Big|^2{\rm d}t\Big)^{\frac{1}{2}}\Big\|_{L^p(\mathbb{R}^n,w)}
\lesssim
\|f\|_{L^p(\mathbb{R}^n,w)}.
\end{eqnarray}
Interpolating the inequalities (2.16) and (2.17) then shows that
\begin{eqnarray*}
\Big\|\Big(\int^2_1\sum_{j\in\mathbb{Z}}\Big|K^j_t*\phi_{j-l}*(S_{m-j}f)(x)\Big|^2{\rm d}t\Big)^{\frac{1}{2}}\Big\|_{L^p(\mathbb{R}^n,\,w)}\lesssim
2^{-t_pm}\|f\|_{L^p(\mathbb{R}^n,w)}.
\end{eqnarray*}
This  gives the desired estimate for ${\rm E}_2$. Combining the eastimates for ${\rm E}_1$ and ${\rm E}_2$ then yields (2.10) for the case $p\in (1,\,q)$ and  $w^{-1/(p-1)}\in A_{p'/q'}(\mathbb{R}^n)$.

We now prove (2.10) for the case of $p\in (q',\,\infty)$ and $w\in A_{p/q'}(\mathbb{R}^n)$. By a standard argument, it suffices to prove that
\begin{eqnarray}M^{\sharp}(\widetilde{\mathcal{M}}_{\Omega}^lf)(x)\lesssim lM_{q'}f(x),\end{eqnarray}
To prove (2.18), let $x\in\mathbb{R}^n$ and $Q$ be a cube containing $x$.
Decompose $f$ as $$f(y)=f(y)\chi_{4nQ}(y)+f(y)\chi_{\mathbb{R}^n\backslash 4nQ}(y)=:f_1(y)+f_2(y).$$
It is obvious that $\widetilde{\mathcal{M}}_{\Omega}^l$ is bounded on $L^{q'}(\mathbb{R}^n)$. Thus,
\begin{eqnarray}\frac{1}{|Q|}\int_{Q}\widetilde{\mathcal{M}}_{\Omega}^lf_1(y){\rm d}y\lesssim\Big(\frac{1}{|Q|}\int_{Q}\big\{\widetilde{\mathcal{M}}_{\Omega}^lf_1(y)\big\}^{q'}{\rm d}y\Big)^{1/q'}\lesssim M_{q'}f(x).\end{eqnarray}
Let $x_0\in Q$ such that $\widetilde{\mathcal{M}}_{\Omega}^lf_2(x_0)<\infty$. For $y\in Q$, it follows from Lemma \ref{l2.2} that
\begin{eqnarray}
&&\Big|\widetilde{\mathcal{M}}_{\Omega}^lf_2(y)-\widetilde{\mathcal{M}}_{\Omega}^lf_2(y_0)\Big|\\
&&\quad\lesssim \Big(\int^2_1\sum_{j\in\mathbb{Z}}\big|K^j_t*\phi_{j-l}*f_2(y)-K^j_t*\phi_{j-l}*f_2(x_0)\big|^2{\rm d}t\Big)^{\frac{1}{2}}\nonumber\\
&&\quad\lesssim\Big(\int^2_1\Big(\sum_{j\in\mathbb{Z}}\big|K^j_t*\phi_{j-l}*f_2(y)-K^j_t*\phi_{j-l}*f_2(x_0)\big|\Big)^2{\rm d}t\Big)^{\frac{1}{2}}\nonumber\\
&&\quad\lesssim lM_{q'}f(x).\nonumber
\end{eqnarray}
Combining the estimates (2.19) and (2.20) leads to that
$$\inf_{c\in\mathbb{C}}\frac{1}{|Q|}\int_{Q}\big|\widetilde{\mathcal{M}}_{\Omega}^lf(y)-c\big|{\rm d}y\lesssim  lM_{q'}f(x)$$
and then establish (2.18).

Finally, we see that (2.10) holds for the case of $p\in (1,\,\infty)$ and $w^{q'}\in A_{p}(\mathbb{R}^n)$, if we invoke the interpolation argument used in the proof of Theorem 2 in \cite{kw}. This completes the proof of Theorem \ref{t2.2}.\qed

\section{Proof of Theorem \ref{t1.2}}
We begin with some preliminary lemmas.
\begin{lemma}\label{l3.1}
Let $\Omega$ be
homogeneous of degree zero and belong to $L^1(S^{n-1})$, $K_t^j$ be defined as in (2.1). Then for
$l\in\mathbb{N}$, $t\in [1,\,2]$, $s\in (1,\,\infty]$, $j_0\in\mathbb{Z}_-$ and  $y\in \mathbb{R}^n$ with $|y|<2^{j_0-4}$,
\begin{eqnarray*}&&\sum_{j>j_0}\sum_{k\in\mathbb{Z}}2^{kn/s}\Big(\int_{2^k<|x|\leq 2^{k+1}}\big|K^j_{t}*\phi_{j-l}(x+y)-K^j_{t}*\phi_{j-l}(x)\big|^{s'}{\rm d}x\Big)^{\frac{1}{s'}}\\
&&\quad\lesssim 2^{l(n+1)}2^{-j_0}|y|.\nonumber\end{eqnarray*}
\end{lemma}
For the proof of Lemma \ref{l3.1}, see \cite{chenhu2}.

For $t\in [1,\,2]$ and $j\in \mathbb{Z}$, let $K_t^j$ be defined as
in (2.1), $\phi$ and $\phi_l$ (with $l\in\mathbb{N}$) be the same as in
Section 2. For $b\in{\rm BMO}(\mathbb{R}^n)$,
let $\widetilde{\mathcal{M}}_{\Omega,\,b}^l$ be the commutator of
$\widetilde{\mathcal{M}}_{\Omega}^{l}$ defined by
$$\widetilde{\mathcal{M}}_{\Omega,\,b}^lf(x)=\Big(\int^2_1\sum_{j\in\mathbb{Z}}\big|F_{j,\,b}^lf(x,\,t)\big|^2{\rm d}t\Big)^{\frac{1}{2}},$$with
$$F_{j,\,b}^lf(x,\,t)=\int_{\mathbb{R}^n}\big(b(x)-b(y)\big)K^j_t*\phi_{j-l}(x-y)f(y)\,{\rm d}y.
$$
For $j_0\in\mathbb{Z}$, define the operator $\widetilde{\mathcal{M}}_{\Omega}^{l,\,j_0}$ by
$$\widetilde{\mathcal{M}}_{\Omega}^{l,\,j_0}f(x)=\Big(\int^2_1\sum_{j\in\mathbb{Z}:j> j_0}\big|F_{j,\,b}^lf(x,\,t)\big|^2{\rm d}t\Big)^{\frac{1}{2}},$$
and the commutator $\widetilde{\mathcal{M}}_{\Omega,\,b}^{l,\,j_0}$
$$\widetilde{\mathcal{M}}_{\Omega,\,b}^{l,\,j_0}f(x)=\Big(\int^2_1\sum_{j\in\mathbb{Z}:j> j_0}\big|F_{j,\,b}^lf(x,\,t)\big|^2{\rm d}t\Big)^{\frac{1}{2}},$$
with $b\in {\rm BMO}(\mathbb{R}^n)$.
\begin{lemma}\label{l3.2}
Let $\Omega$ be homogeneous of degree zero and integrable on $S^{n-1}$. Then for $b\in C^{\infty}_0(\mathbb{R}^n)$, $l\in \mathbb{N}$, $j_0\in\mathbb{Z}_-$,
$$\big|\widetilde{\mathcal{M}}_{\Omega,\,b}^{l,\,j_0}f(x)-\widetilde{\mathcal{M}}_{\Omega,\,b}^lf(x)\big|\lesssim 2^{j_0}MM_{\Omega}f(x).$$
\end{lemma}
\begin{proof} Let $b\in C^{\infty}_0(\mathbb{R}^n)$ with $\|\nabla b\|_{L^{\infty}(\mathbb{R}^n)}=1$. For $t\in [1,\,2]$, by the fact that ${\rm supp}\,K_t^j*\phi_{j-l}\subset
\{x:\,2^{j-2}\leq |x|\leq 2^{j+2}\}$, it is easy to verify that
\begin{eqnarray*}&&\sum_{j\leq j_0}\int_{\mathbb{R}^n}\big|K_{t}^j*\phi_{j-l}(x-y)\big||x-y||f(y)|{\rm d}y\\
&&\quad\lesssim
\sum_{j\leq j_0}\sum_{k\in\mathbb{Z}}2^k\int_{2^k<|x-y|\leq
2^{k+1}}\big|K_{t}^j*\phi_{j-l}(x-y)\big||f(y)|{\rm
d}y\\
&&\quad\lesssim
\sum_{j\leq j_0}\sum_{|k-j|\leq 3}2^k\int_{2^k<|x-y|\leq
2^{k+1}}\big|K_{t}^j*\phi_{j-l}(x-y)\big||f(y)|{\rm
d}y\\
&&\quad\lesssim2^{j_0} M_{\Omega}Mf(x).
\end{eqnarray*}
Thus, \begin{eqnarray*}&&\Big|\widetilde{\mathcal{M}}_{\Omega,\,b}^{l,\,j_0}f(x)-\widetilde{\mathcal{M}}_{\Omega,\,b}^lf(x)\Big|^2\\
&&\quad\leq \sum_{j<j_0}\int^2_1\Big|\int_{\mathbb{R}^n}\big(b(x)-b(y)\big)K^j_t*\phi_{j-l}(x-y)f(y)\Big|^2{\rm d}t\\
&&\quad\lesssim\int^2_1\Big(\sum_{j\leq j_0}\int_{\mathbb{R}^n}|x-y|\big|K_{t}^j*\phi_{j-l}(x-y)f(y)|{\rm
d}y\Big)^2{\rm d}t\\
&&\quad\lesssim \{2^{j_0} M_{\Omega}Mf(x)\}^2.
\end{eqnarray*}
 The desired conclusion now follows immediately.
\end{proof}

Let $p,\,r\in[1,\,\infty)$, $q\in [1,\,\infty]$ and $w$ be a weight,
$L^{p}(L^q([1,\,2]),\,l^r;\,\mathbb{R}^n,\,w)$ be the space of
sequences of functions defined by
$$L^{p}(L^q([1,\,2]),\,l^r;\,\mathbb{R}^n,\,w)=\big\{\vec{f}=\{f_k\}_{k\in \mathbb{Z}}:\, \|\vec{f}\|_
{L^{p}(L^q([1,\,2]),\,l^r;\,\mathbb{R}^n,\,w)}<\infty\big\},$$ with
$$\|\vec{f}\|_{L^{p}(L^q([1,\,2]),\,l^r;\,\mathbb{R}^n,\,w)}=\Big\|\Big(\int^2_1\Big(\sum_{k\in
\mathbb{Z}}|f_k(x,\,t)|^r\Big)^{\frac{q}{r}}{\rm
d}t\Big)^{1/q}\Big\|_{L^{p}(\mathbb{R}^n,\,w)}.$$ With usual
addition and scalar multiplication,
$L^{p}(L^q([1,\,2]),\,l^{r};\,\mathbb{R}^n,\,w)$ is a Banach space.

\begin{lemma}\label{l3.4}
Let $p\in (1,\,\infty)$ and $w\in A_p(\mathbb{R}^n)$,
$\mathcal{G}\subset L^p(L^2([1,\,2]),\,l^{2};\,\mathbb{R}^n,\,w)$.
Suppose that $\mathcal{G}$ satisfies the following five conditions:
\begin{itemize}
\item[\rm (a)] $\mathcal{G}$ is bounded, that is, there exists a constant $C$ such that  for all $\{f_k\}_{k\in\mathbb{Z}}\in
\mathcal{G}$,
$\|\vec{f}\|_{L^p(L^2([1,\,2]),\,l^{2};\,\mathbb{R}^n,\,w)}\leq C$;
\item[\rm (b)] for each fixed $\epsilon>0$, there exists a constant $A>0$, such that for all $\{f_k\}_{k\in\mathbb{Z}}\in
\mathcal{G}$,
$$\Big\|\Big(\int^2_1\sum_{k\in\mathbb{Z}}|f_k(\cdot,\,t)|^2{\rm d}t\Big)^{\frac{1}{2}}\chi_{\{|\cdot|>A\}}(\cdot)\Big\|_{
L^p(\mathbb{R}^n,\,w)}<\epsilon;$$

\item[\rm (c)] for each fixed $\epsilon>0$ and $N\in\mathbb{N}$, there exists a constant $\varrho>0$, such that for all
$\{f_k\}_{k\in\mathbb{Z}}\in \mathcal{G}$,
$$\Big\|\sup_{|h|\leq \varrho}\Big(\int^2_1
\sum_{|k|\leq N}|f_k(x,\,t)-f_k(x+h,\,t)|^2{\rm
d}t\Big)^{\frac{1}{2}}\Big\|_{L^{p}(\mathbb{R}^n,\,w)}< \epsilon;$$
\item[\rm (d)]for each fixed $\epsilon>0$ and $N\in\mathbb{N}$, there exists a constant
$\sigma\in (0,\,1/2)$ such that for all
$\{f_k\}_{k\in\mathbb{Z}}\in \mathcal{G}$,
$$\Big\|\sup_{|s|\leq \sigma}\Big(\int^2_1
\sum_{|k|\leq N}|f_k(\cdot,\,t+s)-f_k(\cdot,\,t)|^2{\rm
d}t\Big)^{\frac{1}{2}}\Big\|_{L^{p}(\mathbb{R}^n,\,w)}< \epsilon,$$
\item[\rm (e)] for each fixed $D>0$ and $\epsilon>0$, there exists $N\in\mathbb{N}$
such that for all $\{f_k\}_{k\in\mathbb{Z}}\in \mathcal {G}$,
$$\Big\|\Big(\int^2_1\sum_{|k|>N}|f_k(\cdot,\,t)|^2{\rm d}t\Big)^{\frac{1}{2}}\chi_{B(0,\,D)}\Big\|_{L^p(\mathbb{R}^n,\,w)}<\epsilon.$$
\end{itemize}
Then $\mathcal{G}$ is  a strongly pre-compact set in
$L^p(L^2([1,\,2]),\,l^{2};\,\mathbb{R}^n,\,w)$.
\end{lemma}
\begin{proof} We employ the argument used in the proof of \cite[Theorem 5]{ccp}, with some refined modifications.
Our goal is to prove that,  for each fixed $\epsilon>0$, there
exists a $\delta=\delta_{\epsilon}>0$ and a mapping
$\Phi_{\epsilon}$ on $L^p(L^2([1,\,2]),\,l^{2};\,\mathbb{R}^n,\,w)$,
such that
$\Phi_{\epsilon}(\mathcal{G})=\{\Phi_{\epsilon}(\vec{f}):\,\vec{f}\in
\mathcal G\}$ is a strong pre-compact set in the space
$L^p(L^2([1,\,2]),\,l^{2};\,\mathbb{R}^n,\,w)$, and for any
$\vec{f}$, $\vec{g}\in \mathcal{G}$,
\begin{eqnarray}&&\|\Phi_{\epsilon}(\vec{f})-\Phi_{\epsilon}(\vec{g})\|_{L^p(L^2([1,2]),l^{2};\,\mathbb{R}^n,w)}<\delta\\
&&\quad\Rightarrow
\|\vec{f}-\vec{g}\|_{L^p(L^2([1,2]),l^{2};\,\mathbb{R}^n,w)}<8\epsilon.\nonumber\end{eqnarray}
If we can prove this, then by Lemma 6 in \cite{ccp}, we see that
$\mathcal{G}$ is a strongly pre-compact set in
$L^p(L^2([1,\,2]),\,l^{2};\,\mathbb{R}^n,\,w)$.

Now let $\epsilon>0$. We choose  $A>1$ large enough as in assumption
(b), $N\in\mathbb{N}$ such that for all $\{f_k\}_{k\in\mathbb{Z}}\in
\mathcal {G}$,
$$\Big\|\Big(\int^2_1\sum_{|k|>N}|f_k(\cdot,\,t)|^2{\rm d}t\Big)^{1/2}\chi_{B(0,\,2A)}
\Big\|_{L^p(\mathbb{R}^n,\,w)}<\epsilon.$$ Let $\varrho\in
(0,\,1/2)$ small enough as in assumption (c) and $\sigma\in
(0,\,1/2)$ small enough such that (d) holds true. Let $Q$ be the
largest cube centered at the origin such that $2Q\subset
B(0,\,\varrho)$, $Q_1,\,\dots,\,Q_J$ be $J$ copies of $Q$ such that
they are non-overlapping, and $\overline{B(0,\,A)}\subset
\overline{\cup_{j=1}^JQ_j}\subset B(0,\,2A)$. Let
$I_1,\,\dots,\,I_L\subset [1,\,2]$ be non-overlapping intervals with
same length $|I|$, such that $|s-t|\leq \sigma$ for all $s,\,t\in
I_j$ $(j=1,\,\dots,\,L)$ and $\cup_{j=1}^NI_j=[1,\,2]$. Define the
mapping $\Phi_{\epsilon}$ on
$L^p(L^2([1,\,2]),\,l^{2};\,\mathbb{R}^n,\,w)$   by
\begin{eqnarray*}\Phi_{\epsilon}(\vec{f})(x,\,t)&=&
\Big\{\dots,0,\,\,\dots,\,0,\,\sum_{i=1}^J\sum_{j=1}^Lm_{Q_i\times I_j}(f_{-N})\chi_{Q_i\times I_j}(x,t),\\
&&\quad\sum_{i=1}^J\sum_{j=1}^Lm_{Q_i\times I_j}(f_{-N+1})\chi_{Q_i\times
I_j}(x,t),\dots,\\
&&\quad\sum_{i=1}^J\sum_{j=1}^Lm_{Q_i\times
I_j}(f_{N})\chi_{Q_i\times I_j}(x,t),0,\dots\Big\},\end{eqnarray*}
where and in the following, $$m_{Q_i\times
I_j}(f_k)=\frac{1}{|Q_i|}\frac{1}{|I_j|}\int_{Q_i\times
I_j}f_k(x,\,t){\rm d}x{\rm d}t.$$

We claim that $\Phi_{\epsilon}$ is bounded on
$L^p(L^2([1,\,2]),\,l^{2};\,\mathbb{R}^n,\,w)$. In fact, if  $p\in
[2,\,\infty)$, we have by the H\"older inequality that
\begin{eqnarray*}|m_{Q_i\times I_j}(f_k)|&\leq &\Big(\frac{1}{|Q_i||I_j|}\int_{I_j\times Q_i}|f_k(y,\,t)|^pw(y)
{\rm d}y{\rm
d}t\Big)^{\frac{1}{p}}\\
&&\quad\times\Big(\frac{1}{|Q_i|}\int_{Q_i}w^{-\frac{1}{p-1}}(y){\rm
d}y\Big)^{\frac{1}{p'}} ,\end{eqnarray*}and
\begin{eqnarray*}
&&\sum_{|k|\leq
N}\Big(\frac{1}{|Q_i||I_j|}\int_{I_j}\int_{Q_i}|f_k(y,\,t)|^pw(y)
{\rm d}y{\rm d}t\Big)^{2/p}\\
&&\quad \lesssim N^{1-2/p}\Big(\sum_{|k|\leq
N}\frac{1}{|Q_i||I_j|}\int_{I_j\times Q_i}|f_k(y,\,t)|^pw(y) {\rm
d}y{\rm d}t\Big)^{2/p}.
\end{eqnarray*}
Therefore,
\begin{eqnarray*}
\|\Phi_{\epsilon}(\vec{f})\|_{L^{p}(L^2([1,\,2]),l^{2};\mathbb{R}^n,w)}^{p}
&=&
\sum_{i=1}^J\sum_{j=1}^L\int_{I_j}\int_{Q_i}\Big(\sum_{|k|\leq
N}|m_{Q_i\times I_j}(f_k)|^2\Big)^{p/2}w(x){\rm d}x{\rm d}t\\
&\lesssim&
N^{p/2-1}\sum_{i=1}^J\sum_{j=1}^L\int_{I_j}\int_{Q_i}\sum_{|k|\leq
N}|f_k(y,t)|^pw(y){\rm d}y{\rm d}t
\\
&\leq&
N^{p/2}\sum_{i=1}^J\sum_{j=1}^L\int_{I_j}\int_{Q_i}\Big\{\sum_{|k|\leq
N}|f_k(y,t)|^2\Big\}^{\frac{p}{2}}w(y){\rm d}y{\rm d}t
\\
&\leq &N^{p/2}\|\vec{f}\|_{L^{p}(L^2([1,\,2]),\,l^2;\,\mathbb{R}^n,\,w)}^{p}.\nonumber
\end{eqnarray*}
On the other hand, for $p\in (1,\,2)$ and $w\in
A_{p}(\mathbb{R}^n)$, we choose $\gamma\in (0,\,1)$ such that $w\in
A_{p-\gamma}(\mathbb{R}^n)$. Note that
\begin{eqnarray*}
&&\sup_{-N\leq k\leq N}\sup_{t\in
[1,\,2]}\Big|\sum_{i=1}^J\sum_{j=1}^Lm_{Q_i\times
I_j}(f_{k})\chi_{Q_i\times I_j}(x,t)\Big|\lesssim\sup_{k\in
\mathbb{Z}}\sup_{t\in [1,\,2]}|f_k(x,\,t)|,
\end{eqnarray*}
which implies that for $p_1\in (1,\,\infty)$,
\begin{eqnarray}\|\Phi_{\epsilon}(\vec{f})\|_{L^{p_1}(L^{\infty}([1,\,2]),\,l^{\infty};\,\mathbb{R}^n,\,w)}\lesssim
\|\vec{f}\|_{L^{p_1}(L^{\infty}([1,\,2]),\,l^{\infty};\,\mathbb{R}^n,\,w)}.\end{eqnarray}
We also have that for $p_0=p-\gamma$,$$|m_{Q_i\times I_j}(f_k)|\leq
\Big(\frac{1}{|Q_i||I_j|}\int_{I_j}\int_{Q_i}|f_k(y,t)|^{p_0}w(y){\rm
d}y{\rm
d}t\Big)^{\frac{1}{p_0}}\Big(\frac{1}{|Q_i|}\int_{Q_i}w^{-\frac{1}{p_0-1}}(y){\rm
d}y\Big)^{\frac{1}{p_0'}} ,$$and so
\begin{eqnarray}\|\Phi_{\epsilon}(\vec{f})\|_{L^{p_0}(L^{p_0}([1,\,2]),\,l^{p_0};\,\mathbb{R}^n,\,w)}\lesssim
\|\vec{f}\|_{L^{p_0}(L^{p_0}([1,\,2]),\,l^{p_0};\,\mathbb{R}^n,\,w)}.\end{eqnarray}
By interpolation, we can deduce from (3.2) and (3.3) that in this
case
\begin{eqnarray*}
\|\Phi_{\epsilon}(\vec{f})\|_{L^{p}(L^2([1,\,2]),\,l^{2};\,\mathbb{R}^n,\,w)}\lesssim\|\vec{f}\|_{L^{p}(L^2([0,\,1]),\,l^2,\,\mathbb{R}^n,\,w)}^{p}.
\end{eqnarray*}
Our claim then follows directly, and so
$\Phi_{\epsilon}(\mathcal{G})=\{\Phi_{\epsilon}(\vec{f}): \vec{f}\in
\mathcal{G}\}$ is  strongly pre-compact  in
$L^p(L^2([1,\,2]),l^{2};\mathbb{R}^n,\,w)$.

We now verify (3.1). Denote $\mathcal{D}=\cup_{i=1}^JQ_i$ and write
\begin{eqnarray*}&&\big\|\vec{f}\chi_{\mathcal{D}}-\Phi_{\epsilon}(\vec{f})
\big\|_{L^p(L^2([1,\,2]),\,l^{2};\,\mathbb{R}^n,\,w)}\\&&\quad\leq
\Big\|\Big(\int^2_1\sum_{|k|\leq
N}\Big|f_k(\cdot,\,t)\chi_{\mathcal{D}}-\sum_{i=1}^{J}\sum_{j=1}^{L}m_{Q_i\times
I_j}
(f_k)\chi_{Q_i\times I_j}(x,t)\Big|^2{\rm d}t\Big)^{1/2}\Big\|_{L^p(\mathbb{R}^n,\,w)}\\
&&\qquad+\Big\|\Big(\int^2_1\sum_{|k|>N}\big|f_k(\cdot,\,t)\big|^2\Big)^{\frac{1}{2}}\chi_{B(0,\,2A)}\Big\|_{L^p(\mathbb{R}^n,\,w)}.
\end{eqnarray*}
Noting that for $x\in Q_i$ with $1\leq i\leq J$,
\begin{eqnarray*}&&\Big\{\int^2_1\sum_{|k|\leq
N}\big|f_k(x,\,t)\chi_{\mathcal{D}}(x)-\sum_{u=1}^{J}\sum_{v=1}^Lm_{Q_u\times
I_v}(f_k)\chi_{Q_u\times I_v}(x,t)\big|^2{\rm
d}t\Big\}^{\frac{1}{2}}\\
&&\quad\lesssim |Q|^{-1/2}|I|^{-1/2}
\Big\{\sum_{j=1}^L\int_{I_j}\int_{Q_i}\int_{I_j}\sum_{|k|\leq
N}\big|f_k(x,\,t)-f_k(y,\,s)\big|^2\,{\rm d}y
{\rm d}s{\rm d}t\Big\}^{\frac{1}{2}}\\
&&\quad\lesssim |Q|^{-1/2}\Big\{\int_{2Q}\int_{1}^2 \sum_{|k|\leq
N}|f_k(x,\,s)-f_k(x+h,\,s)|^2{\rm d}s\,{\rm
d}h\Big\}^{\frac{1}{2}}\\
&&\qquad+ |I|^{-1/2}\Big\{\sum_{j=1}^L\int_{I_j}\int_{I_j}
\sum_{|k|\leq N}|f_k(x,\,t)-f_k(x,\,s)|^2{\rm d}t\,{\rm
d}s\Big\}^{\frac{1}{2}}\\
&&\quad\lesssim\sup_{|h|\leq \varrho}\Big(\int^2_1
\sum_{|k|\leq N}|f_k(x,\,t)-f_k(x+h,\,t)|^2{\rm
d}t\Big)^{\frac{1}{2}}\\
&&\qquad+\sup_{|s|\leq \sigma}\Big(\int^2_1
\sum_{|k|\leq N}|f_k(x,\,t+s)-f_k(x,\,t)|^2{\rm
d}t\Big)^{\frac{1}{2}},
\end{eqnarray*} we then get that
\begin{eqnarray*}\sum_{i=1}^J\int_{Q_i}\Big\{\int^2_1
\sum_{|k|\leq N}\big|f_k(x,\,t)-\sum_{l=1}^{J}m_{Q_l}(f_{k})
\chi_{Q_l}(x)\big|^2\,{\rm d}t\Big\}^{p/2}w(x)\,{\rm d}x \lesssim
2\epsilon.
\end{eqnarray*}
It then follows from the  assumption (b)  that for all $\vec{f}\in
\mathcal{G}$,
\begin{eqnarray*}
\|\vec{f}-\Phi_{\epsilon}(\vec{f})\|_{L^p(L^2([1,\,2]),\,l^{2};\,\mathbb{R}^n,\,w)}&\leq&
\big\|\vec{f}\chi_{\mathcal{D}}-\Phi_{\epsilon}(\vec{f})
\big\|_{L^p(L^2([1,\,2]),\,l^{2};\,\mathbb{R}^n,\,w)}\\
&+&\Big\|\Big(\int^2_1\sum_{k\in\mathbb{Z}}|f_k(\cdot,t)|^2{\rm
d}t\Big)^{\frac{1}{2}}\chi_{\{|\cdot|>A\}}(\cdot)\Big\|_{L^p(\mathbb{R}^n,w)}\\
&<&3\epsilon. \end{eqnarray*} Noting that
\begin{eqnarray*}
\|\vec{f}-\vec{g}\|_{L^p(L^2([1,\,2]),\,l^{2};\,\mathbb{R}^n,\,w)}&\leq&
\|\vec{f}-\Phi_{\epsilon}(\vec{f})\|_{L^p(L^2([1,\,2]),\,l^{2};\,\mathbb{R}^n,\,w)}\\
&&+\|\Phi_{\epsilon}(\vec{f})-\Phi_{\epsilon}(\vec{g})\|_{L^p(L^2([1,\,2]),l^{2};\mathbb{R}^n,w)}\\
&&+\|\vec{g}-\Phi_{\epsilon}(\vec{g})\|_{L^p(L^2([1,\,2]),\,l^{2};\,\mathbb{R}^n,\,w)},\end{eqnarray*}
we then get (3.1) and finish the proof of Lemma \ref{l3.4}.\end{proof}

{\it Proof of Theorem \ref{t1.2}}. Let $j_0\in\mathbb{Z}_-$, $b\in
C^{\infty}_0(\mathbb{R}^n)$ with ${\rm supp}\, b\subset B(0,\,R)$,
$p$ and $w$ be the same as in Theorem \ref{t1.2}. Without loss of generality, we
may assume that $\|b\|_{L^{\infty}(\mathbb{R}^n)}+\|\nabla
b\|_{L^{\infty}(\mathbb{R}^n)}=1.$ We claim that
\begin{itemize}
\item[\rm (i)] for each fixed $\epsilon>0$, there exists a constant $A>0$ such that
$$\Big\|\Big(\int^2_1\sum_{j\in\mathbb{Z}}|F_{j,\,b}^{l}f(x,\,t)|^2{\rm d}t\Big)^{1/2}\chi_{\{|\cdot|>A\}}(\cdot)\Big\|_{L^p(\mathbb{R}^n,\,w)}<\epsilon\|f\|_{L^p(\mathbb{R}^n,\,w)};$$
\item[\rm (ii)]for $s\in (1,\,\infty)$,\begin{eqnarray}&&
\Big(\int^2_1
\sum_{j>j_0}|F_{j,\,b}^{l}f(x,\,t)-F_{j,\,b}^{l}f(x+h,\,t)|^2{\rm
d}t\Big)^{1/2}\\
&&\quad\lesssim
2^{-j_0}|h|\Big(\widetilde{\mathcal{M}}_{\Omega}^{l,\,j_0}f(x)+2^{l(n+1)}M_sf(x)\Big);\nonumber\end{eqnarray}
\item[\rm (iii)] for each $\epsilon>0$ and $N\in\mathbb{N}$, there exists a
constant $\sigma\in (0,\,1/2)$ such that
\begin{eqnarray}&&\Big\|\sup_{|s|\leq \sigma}\Big(\int^2_1
\sum_{|j|\leq
N}|F_{j,b}^{l}f(x,s+t)-F_{j,b}^{l}f(x,t)|^2{\rm
d}t\Big)^{\frac{1}{2}}\Big\|_{L^{p}(\mathbb{R}^n,w)}\\
&&\quad< \epsilon\|f\|_{L^p(\mathbb{R}^n,w)};\nonumber\end{eqnarray}
\item[\rm (iv)]for each fixed $D>0$ and $\epsilon>0$, there exists $N\in\mathbb{N}$
such that
\begin{eqnarray}\Big\|\Big(\int^2_1\sum_{j>N}|F_{j,\,b}^{l}f(\cdot,\,t)|^2{\rm d}t\Big)^{1/2}\chi_{B(0,\,D)}\Big\|_{L^p(\mathbb{R}^n,\,w)}<\epsilon\|f\|_{L^p(\mathbb{R}^n,\,w)}.\end{eqnarray}
\end{itemize}

We now prove claim (i).  Let $t\in [1,\,2]$. For each fixed
$x\in\mathbb{R}^n$ with $|x|>4R$, observe that ${\rm
supp}\,K_t^j*\phi_{j-l}\subset \{2^{j-2}\leq|y|\leq 2^{j+2}\}$, and
$\int_{|z|<R}\big|K_{t}^j*\phi_{j-l}(x-z)\big|{\rm d}z\not =0$ only
if $2^j\approx |x|$. A trivial computation shows that
\begin{eqnarray*}
\int_{|z|<R}\big|K_{t}^j*\phi_{j-l}(x-z)\big|{\rm
d}z&\lesssim&\Big(\int_{|z|<R}\big|K_{t}^j*\phi_{j-l}(x-z)\big|^2{\rm
d}z\Big)^{\frac{1}{2}}R^{\frac{n}{2}}\\
&\lesssim&\Big(\int_{\frac{|x|}{2}\leq
|z|<2|x|}\big|K_{t}^j*\phi_{j-l}(z)\big|^2{\rm
d}z\Big)^{\frac{1}{2}}R^{\frac{n}{2}}\\
&\lesssim&\|K_t^j\|_{L^1(S^{n-1})}\|\phi_{j-l}\|_{L^2(\mathbb{R}^n)}R^{\frac{n}{2}}\\
&\lesssim&2^{nl/2}|x|^{-\frac{n}{2}}R^{\frac{n}{2}}.
\end{eqnarray*}On the other hand,  we have that
\begin{eqnarray*}
&&\sum_{j\in\mathbb{Z}}\Big(\int_{|y|<R}|K_{t}^j*\phi_{j-l}(x-y)||f(y)|^s{\rm d}y\Big)^{\frac{1}{s}}\\
&&\quad=\sum_{j\in\mathbb{Z}:\, 2^j\approx |x|}\Big(\int_{|x|/2\leq |y-x|\leq 2|x|}|K_{t}^j*\phi_{j-l}(x-y)||f(y)|^s{\rm d}y\Big)^{\frac{1}{s}}\\
&&\quad\lesssim\Big(M_{\Omega}M(|f|^s)(x)\Big)^{1/s}.
\end{eqnarray*}
Another application of the H\"older inequality then yields
\begin{eqnarray}\sum_{j\in\mathbb{Z}}|F_{j,\,b}^{l}f(x,\,t)\big|^2&\lesssim&
\sum_{j\in\mathbb{Z}}\Big(\int_{|y|<R}|K_{t}^j*\phi_{j-l}(x-y)||f(y)|^s{\rm d}y\Big)^{2/s}\\
&&\qquad\times\Big(\int_{|y|<R}|K_{t}^j*\phi_{j-l}(x-y)|{\rm d}y\Big)^{2/s'}\nonumber\\
&\lesssim &2^{\frac{nl}{s'}}|x|^{-\frac{n}{s'}}R^{\frac{n}{s'}}\Big(M_{\Omega}M(|f|^s)(x)\Big)^{2/s}.\nonumber
\end{eqnarray}
This, in turn leads to our claim (i).

We turn our attention to claim (ii).  Write
\begin{eqnarray*}|F_{j,\,b}^{l}f(x,\,t)-F_{j,\,b}^{l}f(x+h,\,t)|\leq|b(x)-b(x+h)||F_j^lf(x,\,t)|+{\rm J}^{l}_jf(x,\,t),
\end{eqnarray*}
with$${\rm J}_{j}^{l}f(x,t)=\Big|\int_{\mathbb{R}^n}\big(K_{t}^j*\phi_{j-l}(x-y)-K_{t}^j*\phi_{j-l}(x+h-y)\big)\big(b(x+h)-b(y)\big)f(y){\rm d}y\Big|.$$
It follows from Lemma 3.1 that
\begin{eqnarray*}
\Big(\sum_{j>j_0}|{\rm J}_{j}^{l}f(x,t)|^2\Big)^{\frac{1}{2}}
&\lesssim&\sum_{j>j_0}\int_{\mathbb{R}^n}\big|K_{t}^j*\phi_{j-l}(x-y)-K_{t}^j*\phi_{j-l}(x+h-y)\big||f(y)|{\rm d}y\\
&\lesssim&2^{l(n+1)}|h|2^{-j_0}M_sf(x).
\end{eqnarray*}
Therefore,
\begin{eqnarray*}&&
\Big(\int^2_1
\sum_{j>j_0}|F_{j,\,b}^{l}f(x,\,t)-F_{j,\,b}^{l}f(x+h,\,t)|^2{\rm
d}t\Big)^{\frac{1}{2}}\\
&&\quad\lesssim|h|\widetilde{\mathcal{M}}_{\Omega}^{l,\,j_0}f(x)+2^{l(n+1)}2^{-j_0}|h|M_sf(x).\nonumber
\end{eqnarray*}

We now verify claim (iii). For each fixed $\sigma\in (0,\,1/2)$ and
$t\in [1,\,2]$, let
$$U_{t,\,\sigma}^j(z)=\frac{1}{2^j}\frac{|\Omega(z)|}{|z|^{n-1}}\chi_{\{2^j(t-\sigma)\leq
|z|\leq
2^jt\}}+\frac{1}{2^j}\frac{|\Omega(z)|}{|z|^{n-1}}\chi_{\{2^{j+1}t\leq
|z|\leq 2^{j+1}(t+\sigma)\}},$$ and
$$G_{l,\,t,\,\sigma}^jf(x)=\int_{\mathbb{R}^n}\big(U_{t,\,\sigma}^j*|\phi_{j-l}|\big)(x-y)|f(y)|{\rm
d}y.$$ Note that for $t\in [1,\,2]$,
$$\|U_{t,\,\sigma}^j*|\phi_{j-l}|\|_{L^1(\mathbb{R}^n)}\lesssim
\sigma, \,\,\,\sup_{|j|\leq N}\sup_{t\in
[1,\,2]}|G_{l,\,t,\,\sigma}^jf(x)|\lesssim MM_{\Omega}f(x).$$ Thus,
\begin{eqnarray}
\Big\|\sup_{|j|\leq N}\sup_{t\in
[1,\,2]}|G_{l,\,t,\,\sigma}^jf|\Big\|_{L^{\infty}(\mathbb{R}^n)}\lesssim
\sigma\|f\|_{L^{\infty}(\mathbb{R}^n)},
\end{eqnarray}
and
\begin{eqnarray}
\qquad\Big\|\sup_{|j|\leq N}\sup_{t\in
[1,\,2]}|G_{l,\,t,\sigma}^jf|\Big\|_{L^{p}(\mathbb{R}^n,w)}\lesssim\|MM_{\Omega}\|_{L^{p}(\mathbb{R}^n,w)}\lesssim
\|f\|_{L^{p}(\mathbb{R}^n,\,w)}.
\end{eqnarray}
Interpolating the estimates (3.8) and (3.9) shows that if $p_1\in (p,\,\infty)$,
\begin{eqnarray}
\Big\|\sup_{|j|\leq N}\sup_{t\in
[1,\,2]}|G_{l,\,t,\,\sigma}^jf|\Big\|_{L^{p_1}(\mathbb{R}^n,\,w)}\lesssim
\sigma^{1-p/p_1}\|f\|_{L^{p_1}(\mathbb{R}^n,\,w)}.
\end{eqnarray}
On the other hand, if $p_0\in (1,\,p)$, it then follows from the weighted estimae $M$ and $M_{\Omega}$ that
\begin{eqnarray}
\int_{\mathbb{R}^n}\int^2_1\sum_{|j|\leq
N}\big|G_{l,\,t,\,\sigma}^jf(x)\big|^{p_0}{\rm d}tw(x){\rm d}x \lesssim
N\|f\|_{L^{p_0}(\mathbb{R}^n,\,w)}^{p_0}.
\end{eqnarray}
Choosing $p_1\in (2,\,\infty)$ such that $1/p=1/2+(2-p_0)/(2p_1)$ in (3.10), we get from (3.10) and (3.11) that for $p\in (1,\,2)$,
\begin{eqnarray}
\Big\|\Big(\int^2_1\sum_{|j|\leq
N}\big|G_{l,\,t,\,\sigma}^jf(x)\big|^2{\rm
d}t\Big)^{\frac{1}{2}}\Big\|_{L^p(\mathbb{R}^n,\,w)} \lesssim
N\sigma^{\tau_1}\|f\|_{L^{p}(\mathbb{R}^n,\,w)}.
\end{eqnarray}
with $\tau_1\in (0,\,1)$ a constant. If $p\in [2,\,\infty)$, we obtain from the
Minkowski inequality and the Young inequality that
\begin{eqnarray} &&\Big\|\Big(\int^2_1\sum_{|j|\leq N}|G_{l,\,t,\,\sigma}^jf(x)|^2{\rm
d}t\Big)^{\frac{1}{2}}\Big\|_{L^p(\mathbb{R}^n,\,w)}^2\\
&&\quad\lesssim\Big\{\int_{\mathbb{R}^n}\Big(\int^2_1\Big(\sum_{|j|\leq
N}\int_{\mathbb{R}^n}\big(U_{l,t,\sigma}^j*|\phi_{j-l}|\big)(x-y)|f(y)|{\rm
d}y\Big)^2{\rm d}t\Big)^{\frac{p}{2}}w(x){\rm d}x\Big\}^{\frac{2}{p}}\nonumber\\
&&\quad\lesssim\int^2_1\Big\{\sum_{|j|\leq
N}\Big(\int_{\mathbb{R}^n}\Big(\int_{\mathbb{R}^n}\big(U_{l,\,t,\sigma}^j*|\phi_{j-l}|\big)(x-y)|f(y)|{\rm
d}y\Big)^pw(x){\rm d}x\Big)^{\frac{1}{p}}\Big\}^2{\rm d}t\nonumber\\
&&\quad\lesssim N^2\|f\|_{L^p(\mathbb{R}^n,\,w)}^2.\nonumber
\end{eqnarray}
Also, we have that
\begin{eqnarray} &&\Big\{\int_{\mathbb{R}^n}\Big(\int^2_1\sum_{|j|\leq
N}\Big(\int_{\mathbb{R}^n}\big(U_{l,\,t,\,\sigma}^j*|\phi_{j-l}|\big)(x-y)|f(y)|{\rm
d}y\Big)^2{\rm d}t\Big)^{\frac{p}{2}}{\rm d}x\Big\}^{\frac{2}{p}}\\
&&\quad\lesssim\int^2_1\Big\{\sum_{|j|\leq
N}\big\|\big(U_{l,\,t,\,\sigma}^j*|\phi_{j-l}|*|f|\big\|_{L^p(\mathbb{R}^n)}\Big\}^2{\rm d}t\nonumber\\
&&\quad\lesssim(2N\sigma)^2\|f\|_{L^p(\mathbb{R}^n)}^2,\,\,p\in [2,\,\infty).\nonumber
\end{eqnarray}
The inequalities (3.13) and (3.14), via interpolation with changes of measures, give us that for $p\in [2,\,\infty)$,
\begin{eqnarray}\qquad \Big\|\Big(\int^2_1\sum_{|j|\leq N}|G_{l,\,t,\,\sigma}^jf(x)|^2{\rm
d}t\Big)^{\frac{1}{2}}\Big\|_{L^p(\mathbb{R}^n,\,w)}\lesssim N\sigma^{\tau_2}\|f\|_{L^p(\mathbb{R}^n,\,w)},
\end{eqnarray}
with $\tau_2\in (0,\,1)$ a constant. Since
$$\sup_{|s|\leq \sigma}|F_{j,\,b}^{l}f(x,\,t)-F_{j,\,b}^{l}f(x,\,t+s)|\leq
G_{l,\,t,\,\sigma}^jf(x), $$ our claim (iii) now follow from (3.12)
and (3.15) immediately if we choose $\sigma=\epsilon/(2N)$.

It remains to prove (iv). Let $D>0$ and $N\in\mathbb{N}$ such that
$2^{N-2}>D$. Then for $j> N$ and $x\in\mathbb{R}^n$ with $|x|\leq
D$,
\begin{eqnarray*}\int_{\mathbb{R}^n}\big|K_{t}^j*\phi_{j-l}(x-y)f(y)\big|{\rm d}y&\leq&\int_{\mathbb{R}^n}
\big|K_{t}^j*\phi_{j-l}(x-y)f(y)\big|\chi_{\{|y|\leq 2^{j+3}\}}{\rm d}y\\
&\lesssim&\int_{|y|\leq 2^{j+3}}|f(y)|{\rm d}y\|K_{t}^j\|_{L^1(\mathbb{R}^n)}\|\phi_{j-l}\|_{L^\infty(\mathbb{R}^n)}\nonumber\\
&\lesssim&2^{nl}2^{-nj/p}\|f\|_{L^p(\mathbb{R}^n)}.\nonumber
\end{eqnarray*}Therefore,
\begin{eqnarray}\qquad\Big\|\Big(\int^2_1\sum_{j>N}\big|F_{j,\,b}^{l}f(\cdot,\,t)|^2{\rm d}t\Big)^{\frac{1}{2}}\chi_{B(0,\,D)}\Big\|_{L^p(\mathbb{R}^n)}
\lesssim 2^{nl}\big(\frac{D}{2^N}\big)^{n/p}\|f\|_{L^p(\mathbb{R}^n)}.
\end{eqnarray}
It is obvious that
\begin{eqnarray}\Big\|\Big(\int^2_1\sum_{j>N}\big|F_{j,\,b}^{l}f(\cdot,\,t)|^2{\rm d}t\Big)^{\frac{1}{2}}\chi_{B(0,\,D)}\Big\|_{L^p(\mathbb{R}^n,\,w)}
\lesssim l\|f\|_{L^p(\mathbb{R}^n,\,w)}.
\end{eqnarray}
Interpolating the inequalities (3.16) and (3.17) yields
$$\Big\|\Big(\int^2_1\sum_{j>N}\big|F_{j,\,b}^{l}f(\cdot,\,t)|^2{\rm d}t\Big)^{\frac{1}{2}}\chi_{B(0,\,D)}\Big\|_{L^p(\mathbb{R}^n,\,w)}
\lesssim 2^{\tau_3 nl}\big(\frac{D}{2^N}\big)^{\frac{\tau_3 n}{p}}\|f\|_{L^p(\mathbb{R}^n,\,w)}.
$$
with $\tau_3\in (0,\,1)$ a constant depending only on $w$. The claim (iv) now follows directly.

We can now conclude the proof of Theorem \ref{t1.2}. Let $p\in
(1,\,\infty)$. Note  that
$$\widetilde{\mathcal{M}}_{\Omega,\,b}^{l,\,j_0}f(x)\leq \widetilde{\mathcal{M}}_{\Omega,\,b}^{l}f(x),$$
and so $\widetilde{\mathcal{M}}_{\Omega,\,b}^{l,\,j_0}$ is bounded on $L^p(\mathbb{R}^n,\,w)$. Our claims (i)-(iv), via Lemma \ref{l3.4},
prove that for $b\in
C^{\infty}_0(\mathbb{R}^n)$, $l\in\mathbb{N}$ and $j_0\in\mathbb{Z}_-$, the operator
$\mathcal{F}^{l,\,}_{j_0}$ defined by
\begin{eqnarray}\mathcal{F}^{l}_{j_0}:\, f(x)\rightarrow \{\dots,\,0,\,\dots,\,F_{j_0,\,b}^{l}f(x,\,t),\,F_{j_0+1,\,b}^lf(x,\,t),\,\dots\}\end{eqnarray}
is  compact from $L^p(\mathbb{R}^n,w)$ to
$L^p(L^2([1,\,2]),l^2;\,\mathbb{R}^n,w)$. Thus,
$\widetilde{\mathcal{M}}_{\Omega,\,b}^{l,\,j_0}$ is completely
continuous on $L^p(\mathbb{R}^n,\,w)$. This, via Lemma \ref{l3.2} and Theorem
\ref{t2.2}, shows  that for  $b\in C^{\infty}_0(\mathbb{R}^n)$,
$\widetilde{\mathcal{M}}_{\Omega,\,b}$ is completely continuous on
$L^p(\mathbb{R}^n,\,w)$. Note that
$$\big|\mathcal{M}_{\Omega,\,b}f_k(x)-\mathcal{M}_{\Omega,\,b}f(x)\big|\lesssim \mathcal{M}_{\Omega,\,b}(f_k-f)(x)\lesssim \widetilde{\mathcal{M}}_{\Omega,\,b}(f_k-f)(x).$$
Thus,  for $b\in C^{\infty}_0(\mathbb{R}^n)$,
$\mathcal{M}_{\Omega,\,b}$ is completely continuous on
$L^p(\mathbb{R}^n,\,w)$. Recalling that  $\mathcal{M}_{\Omega,\,b}$ is bounded on
$L^p(\mathbb{R}^n,\,w)$ with bound $C\|b\|_{{\rm BMO}(\mathbb{R}^n)}$,
we obtain that for $b\in {\rm
CMO}(\mathbb{R}^n)$, $\mathcal{M}_{\Omega,\,b}$ is completely
continuous on $L^p(\mathbb{R}^n,\,w)$.\qed

\section{Proof of Theorem \ref{t1.3}}
The following lemma will be useful in the proof of Theorem \ref{t1.3}, and is of independent interest.
\begin{lemma}\label{l4.1}
Let $u\in (1,\,\infty)$, $m\in \mathbb{N}\cup\{0\}$, $S$ be a sublinear operator  which
satisfies that
$$|Sf(x)|\leq \int_{\mathbb{R}^n}|b(x)-b(y)|^m|W(x-y)f(y)|{\rm d}y,$$
with $b\in {\rm BMO}(\mathbb{R}^n)$, and
\begin{eqnarray}\sup_{R>0}R^{n/u}\Big(\int_{R\leq |x|\leq 2R}|W(x)|^{u'}{\rm d}x\Big)^{1/u'}\lesssim 1.\end{eqnarray}
\begin{itemize}
\item[\rm (a)] Let $p\in (u,\,\infty)$, $\lambda\in (0,\,1)$ and $w\in
A_{p/u}(\mathbb{R}^n)$. If $S$ is bounded on $L^p(\mathbb{R}^n,\,w)$
with bound $D\|b\|_{{\rm BMO}(\mathbb{R}^n)}^m$, then for some $\varepsilon\in
(0,\,1)$,$$\big\|Sf\big\|_{L^{p,\,\lambda}(\mathbb{R}^n,\,w)}\lesssim
(D+D^{\varepsilon})\|b\|_{{\rm BMO}(\mathbb{R}^n)}^m\|f\|_{L^{p,\,\lambda}(\mathbb{R}^n,\,w)}.$$
\item[\rm (b)] Let $p\in (1,\,u)$, $w^{r}\in A_1(\mathbb{R}^n)$ for some $r\in (u,\,\infty)$ and $\lambda\in
(0,\,1-r'/u')$. If $S$ is bounded on $L^p(\mathbb{R}^n,\,w)$ with bound $D$,
then for some $\varepsilon\in
(0,\,1)$,$$\big\|Sf\big\|_{L^{p,\,\lambda}(\mathbb{R}^n)}\lesssim(D+
D^{\varepsilon})\|b\|_{{\rm BMO}(\mathbb{R}^n)}^m\|f\|_{L^{p,\,\lambda}(\mathbb{R}^n)}.$$
\end{itemize}
\end{lemma}
\begin{proof} For simplicity, we only consider the case of $m=1$ and $\|b\|_{{\rm BMO}(\mathbb{R}^n)}=1$. For fixed
ball $B$ and $f\in L^{p,\,\lambda}(\mathbb{R}^n,\,w)$, decompose $f$
as
$$f(y)=f(y)\chi_{2B}(y)+\sum_{k=1}^{\infty}f(y)\chi_{2^{k+1}B\backslash
2^kB}(y)=\sum_{k=0}^{\infty}f_k(y).$$ It is obvious  that
\begin{eqnarray*}\int_{B}|Sf_0(y)|^pw(y){\rm
d}y\lesssim D^p\int_{2B}|f(y)|^pw(y)\,{\rm d}y\lesssim
D^p\|f\|_{L^{p,\lambda}(\mathbb{R}^n,w)}^p\{w(B)\}^{\lambda}.\end{eqnarray*}
Let $\theta\in (1,\,p/u)$ such that $w\in A_{p/(\theta u)}(\mathbb{R}^n)$.
For each $k\in\mathbb{N}$, let
$S_kf(x)=S\big(f\chi_{2^{k+1}B\backslash 2^kB}\big)(x)$. Then $S_k$
is also sublinear. We have by the H\"older inequality that for each $x\in B$,
\begin{eqnarray*}
|S_kf(x)|&\lesssim&|b(x)-m_B(b)|\|f_k\|_{L^u(\mathbb{R}^n)}\Big(\int_{2^kB}|W(x-y)|^{u'}{\rm
d}y\Big)^{1/u'}\\
&&+\big\|\big(b-m_B(b)\big)f_k\big\|_{L^u(\mathbb{R}^n)}\Big(\int_{2^kB}|W(x-y)|^{u'}{\rm
d}y\Big)^{1/u'}\\
&\lesssim&|b(x)-m_B(b)|\|f_k\|_{L^p(\mathbb{R}^n,w)}\Big(\int_{2^kB}w^{-\frac{1}{p/u-1}}(y){\rm d}y\Big)^{\frac{1}{u(p/u)'}}|2^kB|^{-\frac{1}{u}}\\
&&+\Big(\int_{2^{k+1}B}|b(y)-m_B(b)|^{p\theta '}{\rm d}y\Big)^{1/(p\theta')}\|f_k\|_{L^p(\mathbb{R}^n,w)}\\
&&\quad\times\Big(\int_{2^kB}w^{-\frac{1}{p/(\theta u)-1}}(y){\rm d}y\Big)^{\frac{1}{u\big(p/(\theta u)\big)'}}|2^kB|^{-\frac{1}{u}},
\end{eqnarray*}
here, $m_B(b)$ denotes the mean value of $b$ on $B$. It follows from the John-Nirenberg inequality that
$$ \Big(\int_{2^{k+1}B}|b(y)-m_B(b)|^{p\theta '}{\rm d}y\Big)^{\frac{1}{p\theta '}}\lesssim k|2^kB|^{\frac{1}{p\theta'}}. $$Therefore, for $q\in (1,\,\infty)$ and  $k\in\mathbb{N}$,  we have
\begin{eqnarray}
\big\|S_kf\big\|_{L^q(B,\,w)}\lesssim k\{w(B)\}^{\frac{1}{q}-\frac{1}{p}}\Big(\frac{w(B)}{w(2^kB)}\Big)^{1/p}\|f_k\|_{L^p(\mathbb{R}^n,\,w)}.
\end{eqnarray}
On the other hand, we deduce from the $L^p(\mathbb{R}^n,\,w)$
boundedness of $S$ that
\begin{eqnarray}
\int_{B}|S_kf(y)|^pw(y)\,{\rm d}y&\lesssim&
D^p\int_{2^kB}|f(x)|^pw(x)\,{\rm d}x
\end{eqnarray}
We then get from (4.2) (with $q=p$) and (4.3) that for $\sigma\in (0,\,1)$,
\begin{eqnarray}
\int_{B}|S_kf(y)|^pw(y)\,{\rm d}y&\lesssim&k^p
D^{p(1-\sigma)}\big(\frac{w(B)}{w(2^kB)}\big)^{\sigma}\int_{2^kB}|f(x)|^pw(x)\,{\rm
d}x
\end{eqnarray}
Recall that $w\in A_{p/u}(\mathbb{R}^n)$. Thus, there exists a constant $\tau\in
(0,\,1)$,
$$\frac{w(B)}{w(2^kB)}\lesssim
\big(\frac{|B|}{|2^kB|}\big)^{\tau},$$see \cite{gra2}.  For fixed $\lambda\in (0,\,1)$, we choose $\sigma$ sufficiently close to $1$ such that $0<\lambda<\sigma$.
It then  follows from (4.4)
that
\begin{eqnarray*}
\sum_{k=0}^{\infty}\Big(\int_{B}|Sf_k(y)|^p{\rm d}y\Big)^{\frac{1}{p}}
&\lesssim
&D^{1-\sigma}\{w(B)\}^{\frac{\lambda}{p}}\sum_{k=0}^{\infty}(k+1)2^{-kn\tau(\sigma-\lambda)/p}
\|f\|_{L^{p,\lambda}(\mathbb{R}^n,w)}\\
&\lesssim&D^{1-\sigma}\{w(B)\}^{\lambda/p}\|f\|_{L^{p,\,\lambda}(\mathbb{R}^n,\,w)}.
\end{eqnarray*}
THis leads to the conclusion (a).

Now we turn our attention to conclusion (b).
From  (4.1), it is obvious  that for $y\in 2^{k+1}B\backslash 2^kB$,
$$\int_{B}|W(x-y)||b(x)-m_B(b)|w(x){\rm d}x\lesssim|2^{k}B|^{-1/u}|B|^{\frac{1}{u\vartheta'}}\Big(\int_{B}w^{u\vartheta}(x){\rm d}x\Big)^{\frac{1}{u\vartheta}}, $$
with $\vartheta\in (1,\,\infty)$ small enough such that $w^{u\vartheta}\in A_1(\mathbb{R}^n)$. This, in turn implies that
\begin{eqnarray*}
&&\int_{B}\int_{2^{k+1}B\backslash 2^kB}|W(x-y)h(y)|{\rm d}y|b(x)-m_B(b)|w(x){\rm d}x\\
&&\quad\lesssim 2^{kn/u'}\frac{w(B)}{w(2^kB)}
\int_{2^{k}B}h(y)w(y)\,{\rm d}y.
\end{eqnarray*}
Therefore, for $s\in (1,\,\infty)$,
\begin{eqnarray}\quad\int_{B}|S_kf(x)|w(x){\rm
d}x&\lesssim&2^{kn/u'}\frac{w(B)}{w(2^kB)}\int_{2^kB}|f(x)|w(x)\,{\rm d}x\\
&&+2^{kn/u'}\frac{w(B)}{w(2^kB)}\int_{2^kB}|f(x)||b(x)-m_B(b)|w(x)\,{\rm d}x\nonumber\\.
&\lesssim&k2^{\frac{kn}{u'}}\frac{w(B)}{w(2^{k+1}B)}\Big(\int_{2^kB}|f(x)|^sw(x){\rm d}x\Big)^{\frac{1}{s}}\big\{w(2^kB)\big\}^{\frac{1}{s'}}.\nonumber\end{eqnarray}
Also, we get by (4.2) that for $q\in (u,\,\infty)$ and $\theta\in (0,\,1)$ with $\theta q\in (u,\,\infty)$,
\begin{eqnarray}
\big\|S_kf\big\|_{L^q(B,\,w)}\lesssim k\{w(B)\}^{\frac{1}{q}-\frac{1}{\theta q}}\Big(\frac{w(B)}{w(2^kB)}\Big)^{\frac{1}{\theta q}}\|f\|_{L^{\theta q}(2^{k+1}B,\,w)}.
\end{eqnarray}
For $p\in (1,\,\infty)$, we choose $q\in (u,\,\infty)$ and $\theta\in (0,\,1)$,  $s\in (1,\,\infty)$ which is close to $1$ sufficiently such that $1/p=t+(1-t)/q$ and $1/p=t/s+(1-t)/(\theta q)$, with $t\in (0,\,1/p)$. By interpolating, we obtain from  the inequalities (4.5) and
(4.6) that
\begin{eqnarray*}\|S_kf\|_{L^p(\mathbb{R}^n,\,w)}\lesssim k2^{\frac{kn}{pu'}}\Big(\frac{w(B)}{w(2^kB)}\Big)^{1/p}\|f\|_{L^p(2^kB,\,w)}.
\end{eqnarray*}
The fact that $w^{r}\in A_1(\mathbb{R}^n)$ tells us that
$$\frac{w(B)}{w(2^kB)}\lesssim 2^{-kn(r-1)/r},$$ see \cite[p. 306]{gra2}.
This, together with the fact that $S$ is bounded on $L^p(\mathbb{R}^n,\,w)$ with bound
$D$, gives us that for any $\omega\in (0,\,1)$,
\begin{eqnarray*}\Big(\int_{B}|S_kf(x)|^pw(x){\rm
d}x\Big)^{1/p}&\lesssim&
D^{1-\omega}k2^{\frac{\omega kn}{pu'}}\Big(\frac{w(B)}{w(2^kB)}\Big)^{\omega/p}\|f\|_{L^p(2^kB,\,w)}\\
&\lesssim&\{w(B)\}^{\lambda/p}
D^{1-\omega}k2^{\frac{kn}{p}\big(\frac{\omega}{u'}-\frac{\omega-\lambda}{r'}\big)}\|f\|_{L^{p,\lambda}(\mathbb{R}^n,w)}.
\end{eqnarray*}
For fixed $\lambda\in (0,\,1-r'/u')$, we choose $\omega\in (\lambda,\,1)$
sufficiently close to $1$ such that $\omega/u'-(\omega-\lambda)/r'<0$.
Summing over the last inequality yields conclusion (b).
\end{proof}

Let $p,\,r\in[1,\,\infty)$, $\lambda\in (0,\,1)$, $q\in [1,\,\infty]$ and $w$ be a weight.
Define the space $L^{p,\,\lambda}(L^q([1,\,2]),\,l^r;\,\mathbb{R}^n,\,w)$  by
$$L^{p,\,\lambda}(L^q([1,\,2]),\,l^r;\,\mathbb{R}^n,\,w)=\big\{\vec{f}=\{f_k\}_{k\in \mathbb{Z}}:\, \|\vec{f}\|_
{L^{p,\,\lambda}(L^q([1,\,2]),\,l^r;\,\mathbb{R}^n,\,w)}<\infty\big\},$$ with
$$\|\vec{f}\|_{L^{p,\,\lambda}(L^q([1,\,2]),\,l^r;\,\mathbb{R}^n,\,w)}=\Big\|\Big(\int^2_1\Big(\sum_{k\in
\mathbb{Z}}|f_k(x,\,t)|^r\Big)^{\frac{q}{r}}{\rm
d}t\Big)^{\frac{1}{q}}\Big\|_{L^{p,\,\lambda}(\mathbb{R}^n,\,w)}.$$ With usual
addition and scalar multiplication,
$L^{p,\,\lambda}(L^q([1,\,2]),\,l^{r};\,\mathbb{R}^n,\,w)$ is a Banach space.

\begin{lemma}\label{l4.2}
Let $p\in (1,\,\infty)$, $\lambda\in (0,\,1)$ and $w\in
A_p(\mathbb{R}^n)$, $\mathcal{G}$ be a subset in $
L^{p,\,\lambda}(L^2([1,\,2]),\,l^{2};\,\mathbb{R}^n,\,w)$. Suppose that
$\mathcal{G}$ satisfies the following five conditions:
\begin{itemize}
\item[\rm (a)] $\mathcal{G}$ is a bounded set in $L^{p,\,\lambda}(L^2([1,\,2]),\,l^{2};\,\mathbb{R}^n,\,w)$;
\item[\rm (b)] for each fixed $\epsilon>0$, there exists a constant $A>0$, such that for all $\{f_k\}_{k\in\mathbb{Z}}\in
\mathcal{G}$,
$$\Big\|\Big(\int^2_1\sum_{k\in\mathbb{Z}}|f_k(\cdot,\,t)|^2{\rm d}t\Big)^{\frac{1}{2}}\chi_{\{|\cdot|>A\}}(\cdot)\Big\|_{
L^{p,\,\lambda}(\mathbb{R}^n,\,w)}<\epsilon;$$

\item[\rm (c)] for each fixed $\epsilon>0$ and $N\in\mathbb{N}$, there exists a constant $\varrho>0$, such that for all
$\vec{f}=\{f_k\}_{k\in\mathbb{Z}}\in \mathcal{G}$,
$$\Big\|\sup_{|h|\leq \varrho}\Big(\int^2_1
\sum_{|k|\leq N}|f_k(\cdot,\,t)-f_k(\cdot+h,\,t)|^2{\rm
d}t\Big)^{\frac{1}{2}}\Big\|_{L^{p,\,\lambda}(\mathbb{R}^n,\,w)}< \epsilon;$$
\item[\rm (d)]for each fixed $\epsilon>0$ and $N\in\mathbb{N}$, there exists a constant
$\sigma\in (0,\,1/2)$ such that for all
$\vec{f}=\{f_k\}_{k\in\mathbb{Z}}\in \mathcal{G}$,
$$\Big\|\sup_{|s|\leq \sigma}\Big(\int^2_1
\sum_{|k|\leq N}|f_k(\cdot,\,t+s)-f_k(\cdot,\,t)|^2{\rm
d}t\Big)^{\frac{1}{2}}\Big\|_{L^{p,\lambda}(\mathbb{R}^n,\,w)}< \epsilon,$$
\item[\rm (e)] for each fixed $D>0$ and $\epsilon>0$, there exists $N\in\mathbb{N}$
such that for all $\{f_k\}_{k\in\mathbb{Z}}\in \mathcal {G}$,
$$\Big\|\Big(\int^2_1\sum_{|k|>N}|f_k(\cdot,\,t)|^2{\rm d}t\Big)^{\frac{1}{2}}\chi_{B(0,\,D)}\Big\|_{L^{p,\,\lambda}(\mathbb{R}^n,\,w)}<\epsilon.$$
\end{itemize}
Then $\mathcal{G}$ is  strongly pre-compact in
$L^{p,\,\lambda}(L^2([1,\,2]),\,l^{2};\,\mathbb{R}^n,\,w)$.
\end{lemma}
\begin{proof}The proof is similar to the proof of Lemma \ref{l3.4}, and so we only give the outline here. It suffices to prove that, for each fixed $\epsilon>0$, there exists a
$\delta=\delta_{\epsilon}>0$ and a mapping $\Phi_{\epsilon}$ on
$L^{p,\,\lambda}(L^2([1,\,2]),\,l^{2};\,\mathbb{R}^n,\,w)$, such that
$\Phi_{\epsilon}(\mathcal{G})=\{\Phi_{\epsilon}(\vec{f}):\,\vec{f}\in
\mathcal G\}$ is a strongly pre-compact set in
 $L^{p,\,\lambda}(L^2([1,\,2]),\,l^{2};\,\mathbb{R}^n,\,w)$, and  for any $\vec{f}$, $\vec{g}\in \mathcal{G}$,
$$\|\Phi_{\epsilon}(\vec{f})-\Phi_{\epsilon}(\vec{g})\|_{L^{p,\,\lambda}(L^2([1,\,2]),\,l^{2};\,\mathbb{R}^n,\,w)}<\delta\Rightarrow
\|\vec{f}-\vec{g}\|_{L^{p,\,\lambda}(L^2([1,\,2]),\,l^{2};\,\mathbb{R}^n,\,w))}<8\epsilon.$$

For fixed $\epsilon>0$, we choose  $A>1$
large enough as in assumption (b),  and $N\in\mathbb{N}$ such that for all
$\{f_k\}_{k\in\mathbb{Z}}\in \mathcal {G}$,
$$\Big\|\Big(\int^2_1\sum_{|k|>N}|f_k(\cdot,\,t)|^2{\rm d}t\Big)^{\frac{1}{2}}\chi_{B(0,\,2A)}\Big\|_{L^{p,\,\lambda}(\mathbb{R}^n,\,w)}<\epsilon.$$
Let $Q$, $Q_1,\,\dots,\,Q_J$, $\mathcal{D}$,
$I_1,\,\dots,\,I_L\subset [1,\,2]$, and $\Phi_{\epsilon}$ be the same as in the proof of Lemma \ref{l3.2}. For such fixed $N$, let $\varrho$ and $\sigma\in (0,\,1/2)$ small enough such that for all
$\vec{f}=\{f_k\}_{k\in\mathbb{Z}}\in \mathcal{G}$,
\begin{eqnarray}\Big\|\sup_{|h|\leq \varrho}\Big(\int^2_1
\sum_{|k|\leq N}|f_k(\cdot,\,t)-f_k(\cdot+h,\,t)|^2{\rm
d}t\Big)^{\frac{1}{2}}\Big\|_{L^{p,\,\lambda}(\mathbb{R}^n,\,w)}< \frac{\epsilon}{2J};\end{eqnarray}
\begin{eqnarray}\Big\|\sup_{|s|\leq \sigma}\Big(\int^2_1
\sum_{|k|\leq N}|f_k(\cdot,\,t+s)-f_k(\cdot,\,t)|^2{\rm
d}t\Big)^{\frac{1}{2}}\Big\|_{L^{p,\lambda}(\mathbb{R}^n,\,w)}< \frac{\epsilon}{2J},\end{eqnarray}
We can verify that $\Phi_{\epsilon}$ is bounded on $L^{p,\,\lambda}(L^2([1,\,2]),\,l^{2};\,\mathbb{R}^n,\,w)$, and consequently,
$\Phi_{\epsilon}(\mathcal{G})=\{\Phi_{\epsilon}(\vec{f}):\,
\vec{f}\in \mathcal{G}\}$ is a strongly pre-compact set in
$L^{p,\lambda}(L^2([1,\,2]),l^{2};\,\mathbb{R}^n,w)$. Recall that for $x\in Q_i$ with $1\leq i\leq J$,
\begin{eqnarray*}&&\Big\{\int^2_1\sum_{|k|\leq
N}\big|f_k(x,\,t)-\sum_{v=1}^Lm_{Q_i\times
I_v}(f_k)\chi_{I_v}(t)\big|^2{\rm
d}t\Big\}^{\frac{1}{2}}\\
&&\quad\lesssim\sup_{|h|\leq \varrho}\Big(\int^2_1
\sum_{|k|\leq N}|f_k(x,\,t)-f_k(x+h,\,t)|^2{\rm
d}t\Big)^{\frac{1}{2}}\\
&&\qquad+\sup_{|s|\leq \sigma}\Big(\int^2_1
\sum_{|k|\leq N}|f_k(x,\,t+s)-f_k(x,\,t)|^2{\rm
d}t\Big)^{\frac{1}{2}}.
\end{eqnarray*} For a ball
$B(y,\,r)$,  a trivial computation involving (4.7) and (4.8), leads to that
\begin{eqnarray*}
&&\int_{B(y,\,r)}\Big(\int^2_1\sum_{|k|\leq
N}\big|f_k(x,\,t)\chi_{\mathcal{D}}-\sum_{i=1}^{J}\sum_{j=1}^{L}m_{Q_i\times
I_j}
(f_k)\chi_{Q_i\times I_j}(x,t)\big|^2{\rm d}t\Big)^{\frac{p}{2}}w(x){\rm d}x\\
&&\quad=\sum_{i=1}^J\int_{B(y,\,r)\cap Q_i}\Big(\int^2_1\sum_{|k|\leq
N}\big|f_k(x,\,t)-\sum_{i=1}^{J}m_{Q_i\times
I_j}
(f_k)\chi_{I_j}(t)\big|^2{\rm d}t\Big)^{\frac{p}{2}}w(x){\rm d}x\\
&&\quad\lesssim \epsilon\{w(B(y,\,r))\}^{\lambda}.
\end{eqnarray*}
Therefore,
\begin{eqnarray*}&&\int_{B(y,\,r)}\big\|\vec{f}\chi_{\mathcal{D}}-\Phi_{\epsilon}(\vec{f})
\big\|_{L^2([1,\,2]),\,l^{2})}^pw(x)
{\rm d}x\\
&& \lesssim
\int_{B(y,\,r)}\Big(\int^2_1\sum_{|k|\leq
N}\big|f_k(x,\,t)\chi_{\mathcal{D}}-\sum_{i=1}^{J}\sum_{j=1}^{L}m_{Q_i\times
I_j}
(f_k)\chi_{Q_i\times I_j}(x,t)\big|^2{\rm d}t\Big)^{\frac{p}{2}}w(x){\rm d}x\\
&&\quad+\int_{B(y,\,r)}\Big(\int^2_1\sum_{|k|>N}\big|f_k(x,\,t)\big|^2\Big)^{p/2}\chi_{B(0,\,2A)}(x)
w(x){\rm d}x\\
&&\lesssim 2\epsilon\{w(B(y,\,r))\}^{\lambda}.
\end{eqnarray*}
It then follows from the  assumption (b)  that for all $\vec{f}\in
\mathcal{G}$,
\begin{eqnarray*}
\|\vec{f}-\Phi_{\epsilon}(\vec{f})\|_{L^{p,\,\lambda}(L^2([1,\,2]),\,l^{2};\,\mathbb{R}^n,\,w)}&\leq &
\big\|\vec{f}\chi_{\mathcal{D}}-\Phi_{\epsilon}(\vec{f})
\big\|_{L^p(L^2([1,\,2]),\,l^{2};\,\mathbb{R}^n,\,w)}+\epsilon\\
&<&3\epsilon, \end{eqnarray*}and
$$\|\vec{f}-\vec{g}\|_{L^{p,\,\lambda}(\mathbb{R}^n)}< 6\epsilon+\|\Phi_{\epsilon}(f)-\Phi_{\epsilon}(\vec{g})\|_{L^{p,\,\lambda}(\mathbb{R}^n)}.$$
This completes the proof of Lemma \ref{l4.2}.
\end{proof}
{\it Proof of Theorem \ref{t1.3}}. We only consider the case of $p\in (q',\,\infty)$, $w\in A_{p/q'}(\mathbb{R}^n)$ and $\lambda\in (0,\,1)$. Recall that $\mathcal{M}_{\Omega,b}$ is bounded on $L^{p}(\mathbb{R}^n,\,w)$. Thus, by Lemma 4.2, we know that $\mathcal{M}_{\Omega,\,b}$ is bounded on $L^{p,\,\lambda}(\mathbb{R}^n,\,w)$. Thus, it suffices to prove that for $b\in C^{\infty}_0(\mathbb{R}^n)$, $\mathcal{M}_{\Omega,\,b}$ is completely continuous on $L^{p,\,\lambda}(\mathbb{R}^n,\,w)$.

Let $j_0\in\mathbb{Z}_-$, $b\in
C^{\infty}_0(\mathbb{R}^n)$ with ${\rm supp}\, b\subset B(0,\,R)$ and $\|b\|_{L^{\infty}(\mathbb{R}^n)}+\|\nabla
b\|_{L^{\infty}(\mathbb{R}^n)}=1.$ Let $\widetilde{K^j}(z)=\frac{|\Omega(z)|}{|z|^n}\chi_{\{2^{j-1}\leq |z|\leq 2^{j+2}\}}(z).$
By the Minkowski inequality,
\begin{eqnarray*}\Big(\int^2_1\sum_{j\in\mathbb{Z}}|F_{j,\,b}^{l}f(x,\,t)|^2{\rm d}t\Big)^{\frac{1}{2}}&\leq & \Big(\sum_{j\in\mathbb{Z}}\int^2_1|F_{j,\,b}^{l}f(x,\,t)|^2{\rm d}t\Big)^{\frac{1}{2}}\\
&\lesssim&\sum_{j\in\mathbb{Z}}\int_{\mathbb{R}^n}\widetilde{K^j}*\phi_{j-l}(x-y)|f(y)|\,{\rm d}y.
\end{eqnarray*}
It is obvious that ${\rm supp}\,\widetilde{K^j}*\phi_{j-l}\subset \{x:\, 2^{j-3}\leq |x|\leq 2^{j+3}\}$, and for any $R>0$,
$$\int_{R\leq |x|\leq 2R}\Big|\sum_{j\in\mathbb{Z}}\widetilde{K^j}*\phi_{j-l}(x)\Big|^q{\rm d}x\leq \sum_{j:\, 2^j\approx R}\|\widetilde{K^j}*\phi_{j-l}\|_{L^q(\mathbb{R}^n)}^q\lesssim R^{-nq+n}.$$
Let $\epsilon>0$. We deduce from  Lemma \ref{l4.1} and the inequality (3.7) that,  there exists a constant $A>0$, such that for all $f\in L^p(\mathbb{R}^n)$,
\begin{eqnarray}\quad\Big\|\Big(\int^2_1\sum_{j\in\mathbb{Z}}|F_{j,\,b}^{l}f(x,\,t)|^2{\rm d}t\Big)^{\frac{1}{2}}\chi_{\{|\cdot|>A\}}(\cdot)\Big\|_{L^{p,\,\lambda}(\mathbb{R}^n,\,w)}<\epsilon\|f\|_{L^{p,\,\lambda}(\mathbb{R}^n,\,w)}.
\end{eqnarray}
Recall that $\widetilde{\mathcal{M}}_{\Omega}^{l,\,j_0}$ is bounded on $L^{p,\,\lambda}(\mathbb{R}^n,\,w)$. For $r>1$ small enough, $M_r$ is also bounded on $L^{p,\,\lambda}(\mathbb{R}^n,\,w)$ (see \cite{ks}).  Thus by (3.4), we know that there exists  a constant $\varrho>0$, such that
\begin{eqnarray}&&
\Big\|\sup_{|h|\leq \varrho}\Big(\int^2_1
\sum_{j>j_0}|F_{j,\,b}^{l}f(\cdot,t)-F_{j,\,b}^{l}f(\cdot+h,t)|^2{\rm
d}t\Big)^{\frac{1}{2}}\Big\|_{L^{p,\lambda}(\mathbb{R}^n,\,w)}\\
&&\quad\lesssim \epsilon\|f\|_{L^{p,\lambda}(\mathbb{R}^n,\,w)}.\nonumber\end{eqnarray}
It follows from Lemma \ref{l4.1}, estimate (3.5) that
for each $N\in\mathbb{N}$, there exists a
constant $\sigma\in (0,\,1/2)$ such that
\begin{eqnarray}&&\Big\|\sup_{|s|\leq \sigma}\Big(\int^2_1
\sum_{|j|\leq
N}|F_{j,\,b}^{l}f(\cdot,\,s+t)-F_{j,\,b}^{l}f(\cdot,\,t)|^2{\rm
d}t\Big)^{\frac{1}{2}}\Big\|_{L^{p,\lambda}(\mathbb{R}^n,w)}\\
&&\quad< \epsilon\|f\|_{L^{p,\lambda}(\mathbb{R}^n,w)}.\nonumber\end{eqnarray}
We also obtain by Lemma \ref{l4.1} and (3.6) that
for each fixed $D>0$, there exists $N\in\mathbb{N}$
such that
\begin{eqnarray}\Big\|\Big(\int^2_1\sum_{j>N}|F_{j,\,b}^{l}f(\cdot,\,t)|^2{\rm d}t\Big)^{\frac{1}{2}}\chi_{B(0,\,D)}\Big\|_{L^{p,\lambda}(\mathbb{R}^n,\,w)}<\epsilon\|f\|_{L^{p,\lambda}(\mathbb{R}^n,\,w)}.\end{eqnarray}
The inequalities (4.9)-(4.12), via Lemma 4.2, tell us for any  $j_0\in \mathbb{Z}_-$, the operator $\mathcal{F}^{l}_{j_0}$ defined by (3.18) is compact from $L^{p,\,\lambda}(\mathbb{R}^n,\,w)$ to
$L^{p,\,\lambda}(L^2([1,\,2]),\,l^2;\,\mathbb{R}^n,\,w)$.
On the other hand, by Lemma \ref{l4.1}, Theorem \ref{t2.2} and Lemma \ref{l3.2}, we know that
$$\|\widetilde{\mathcal{M}}_{\Omega}f-\widetilde{\mathcal{M}}_{\Omega}^{l}f\|_{L^{p,\,\lambda}(\mathbb{R}^n,\,w)}\lesssim
2^{-\varepsilon\varrho_pl}\|f\|_{L^{p,\,\lambda}(\mathbb{R}^n,\,w)},$$
and
$$\big\|\widetilde{\mathcal{M}}_{\Omega,\,b}^{l,\,j_0}f-\widetilde{\mathcal{M}}_{\Omega,\,b}^lf\big\|_{L^{p,\,\lambda}(\mathbb{R}^n,\,w)}\lesssim 2^{\varepsilon j_0}\|f\|_{L^{p,\,\lambda}(\mathbb{R}^n,\,w)}.$$
As it was shown in the proof of Theorem \ref{t1.2},   we can deduce from the last facts that $\mathcal{M}_{\Omega,\,b}$ is completely continuous on $L^{p,\,\lambda}(\mathbb{R}^n,\,w)$ when $b\in C_0^{\infty}(\mathbb{R}^n)$. This completes the proof of Theorem \ref{t1.3}.
\qed

\end{document}